\documentclass[11pt]{article}

\usepackage{graphicx,psfrag}

\usepackage{amssymb,amsmath,amsthm,enumerate}
\usepackage{fullpage}
\usepackage{comment}

\usepackage{amsmath,amsthm,amsfonts,amssymb,epsfig,verbatim, fge}
\usepackage[normalem]{ulem}
\usepackage{xcolor}
\usepackage{circledsteps}
\usepackage{url}
\usepackage{breqn}
\usepackage{bbm}
\usepackage[colorlinks, citecolor=blue]{hyperref}

\newtheorem{thm}{Theorem}[section]
\newtheorem{lem}[thm]{Lemma}
\newtheorem{prop}[thm]{Proposition}
\newtheorem{cor}[thm]{Corollary}

\newtheorem*{rem}{Remark}

\numberwithin{equation}{section}

\newcommand{\E}{\mathbf{E}}

\newcommand{\M}{\mathcal{M}}

\newcommand{\prob}{\mathbf{P}}

\newcommand{\R}{\mathbb{R}}

\newcommand{\Z}{\mathbb{Z}}

\newcommand{\la}{\langle}
\newcommand{\ra}{\rangle}

\DeclareMathOperator{\Var}{Var}
\DeclareMathOperator{\Cov}{Cov}

\newcommand{\bB}{\mathbb{B}}
\renewcommand{\M}{\mathsf{M}}
\newcommand{\calF}{\mathcal{F}}
\newcommand{\eps}{\epsilon}
\renewcommand{\setminus}{\mathbin{\fgebackslash}}
\newcommand{\1}{\mathbf{1}}
\newcommand{\gE}{\mathsf{E}}
\newcommand{\gVar}{\mathsf{Var}}
\newcommand{\gH}{\mathsf{H}}

\begin{document}
	\title{Central Limit Theorem in Disordered Monomer-Dimer Model }
	\author{Wai-Kit Lam\thanks{National Taiwan University. Email: waikitlam@ntu.edu.tw. The research of W.-K. L. is supported by National Science and Technology Council in Taiwan Grant 110-2115-M002-012-MY3 and NTU New Faculty Founding Research Grant NTU-111L7452.} 
		\and Arnab Sen\thanks{University of Minnesota. Email: arnab@umn.edu. The research of A.-S. is partly supported by Simons Foundation MP-TSM-00002716}}
	\date{}
	\maketitle
	
\renewcommand{\thefootnote}{\fnsymbol{footnote}} 

\begin{abstract}
We consider the disordered monomer-dimer model on general finite graphs with bounded degrees. Under the finite fourth moment assumption on the weight distributions, we prove a Gaussian central limit theorem for the free energy of the associated Gibbs measure with a rate of convergence. The central limit theorem continues to hold under a nearly optimal finite $(2+\eps)$-moment assumption on the weight distributions if the underlying graphs are further assumed to have a uniformly subexponential volume growth. This generalizes a recent result by Dey and Krishnan \cite{DeyKrishnan} who showed a Gaussian central limit theorem in the disordered monomer-dimer model on cylinder graphs. Our proof relies on the idea that the disordered monomer-dimer model exhibits a decay of correlation with high probability. We also establish a central limit theorem for the Gibbs average of the number of dimers where the underlying graph has subexponential volume growth and the edge weights are Gaussians.

\end{abstract}

\section{Introduction}
The monomer-dimer model was introduced as a simple yet effective model in condensed-matter physics that describes the absorption of monoatomic (monomer) or diatomic (dimer) on certain surfaces \cite{roberts, roberts1938some, chang1939statistical, chang1939number}. The dimers occupy the pair of adjacent sites (or edges) of a graph, whereas the monomers occupy the rest of the vertices. The key feature of the model is the hard-core interaction among the dimers that excludes two dimers to share a common vertex. In the language of graph theory, the set of dimers forms a (partial) matching of the graph and the monomers can be regarded as the unmatched vertices. The monomer-dimer model is the Gibbs measure on the space of dimer configurations or matching on a finite graph where the energy of a matching is given by the sum of the weights of the edges belonging to that matching (dimer activity) and the weights of the unmatched vertices (monomer activity).

In a seminal paper, Heilmann and Lieb \cite{HeilmannLieb} established that the monomer-dimer model does not have a phase transition by studying the zeroes of the partition function. If the monomer weights are taken to be a constant $\nu$, then the partition function can be viewed as a polynomial in $e^\nu$. Heilmann and Lieb showed that the all zeroes of this polynomial, known as the matching polynomial, lie on the  imaginary axis. This  makes the free energy an analytic function of $\nu$, which implies the absence of any phase transition. This special localization property of the zeroes of the partition function also yields a central limit theorem for the number of edges in a random matching for a large class of finite graphs \cite{godsil1981matching, kahn2000normal, lebowitz2016central}.


%
One can naturally inject randomness into the monomer-dimer model by considering the underlying graph to be random or (and) by taking the weights to be random. For a sequence of finite graphs converging locally to a random tree, the limiting free energy of the monomer-dimer model with constant weights can be computed via so-called cavity method and the solution turns out to be a function of the unique fixed point of some distributional recursion relation \cite{zdeborova2006number,alberici2014solution, matching_infinite}. In \cite{alberici2015mean}, Alberici et.\ al.\ considered the monomer-dimer model on a complete graph with constant edge weights and i.i.d.\ random vertex weights. Exploiting a Gaussian representation of the partition function, they provided a variational formula for the limiting free energy. 

Recently, Dey and Krishnan \cite{DeyKrishnan} studied the monomer-dimer model with  both vertex and edge weights being random (referred to by them as the disordered monomer-dimer model) on cylinder graphs (i.e., the cartesian product of a path graph on $n$ vertices, and a fixed graph). Among other results, they proved that the free energy has Gaussian fluctuation. The proof of their central limit theorem is based on the dyadic decomposition of the free energy into independent components which heavily relies on the one-dimensional nature of the underlying graph. Such decomposition does not extend immediately to the higher dimensional lattices.  The aim of our paper is to show that the free energy of the disordered monomer-dimer model on any bounded degree finite graph is asymptotically normal  under a mild moment condition on the weights. Moreover, an explicit error bound in the Kolmogorov-Smirnov distance is provided. 
Our proof hinges on establishing the correlation decay in the disordered model. In the constant weight case, van den Berg  \cite{vdB} introduced a geometric argument involving disagreement percolation to establish decay of correlation and applied it to deduce that the model has no phase transition. We adapt van den Berg's technique in the random weights setting to obtain the necessary correlation decay result. To show the asymptotic normality of the free energy, we then invoke a central limit result by Chatterjee \cite{chatterjee2008new} based on a generalized perturbative approach to Stein’s method.

\subsection{The model and main results}
Let $G = (V(G), E(G))$ be a finite graph. A (partial) matching $M$ of $G$ is a collection of disjoint edges from $E(G)$ such that no pair of them is incident on the same vertex. We denote the collection of all matchings on $G$ by $\mathcal{M}_G$.

We assign i.i.d.\ weights $(w_e)_{e\in E(G)}$ to the edges and i.i.d.\ weights $(\nu_x)_{x \in V(G)}$ to the vertices of $G$. The edge weights are independent of the vertex weights, but their distributions are possibly different. For $x \in V(G)$ and $M \in \mathcal{M}_G$, we write $x \not \in M$ to indicate that the vertex $x$ is unmatched in the matching $M$.  Given the weights, the disordered monomer-dimer model on $G$ is defined by 
the following random Gibbs measure $\mu_G$.
\[
\mu_G(M)  = Z_G^{-1} \exp\Big ( \beta \Big(\sum_{e \in M} w_e + \sum_{x \not\in M} \nu_x \Big)\Big), \quad M \in \mathcal{M}_G, 
\]
where $\beta > 0$ is the inverse temperature and $Z_G$ is the normalization constant, called the partition function, which is  given by
\[ 
Z_G = \sum_{M \in \mathcal{M}_G}  \exp\Big ( \beta \Big(\sum_{e \in M} w_e + \sum_{x \not\in M} \nu_x\Big)\Big). \]
Note that the inverse temperature $\beta$ can be absorbed into the weights. So, henceforth, we will assume, without loss of generality,  that $\beta = 1$. 

Our main result establishes a central limit theorem of the free energy (or pressure) of the monomer-dimer model which is defined as $F = \log{Z_G}$.  Understanding the fluctuation of the free energy is a naturally important question in disordered models in statistical physics. Bounds on the fluctuation of free energy was a key ingredient to the Aizenman-Wehr's famous proof \cite{AizenmanWehr}  of the absence of correlation in the two-dimensional random field Ising model at any field strength (see also \cite[Section~7.2]{Bovier})). Later Chatterjee proved \cite{chatterjee2019central} that the free energy of the random field Ising model with all plus or all minus boundary conditions obeys a central limit theorem at any temperature (including zero temperature) and any dimension. In \cite{wehr1990fluctuations}, Aizenman and Wehr discussed how the fluctuation of free energy yields some inequality between characteristic exponents for directed polymers in a random environment in $d$-dimensional lattice. Of course, the central limit theorem provides a more precise asymptotic of the fluctuation. The central limit theorem of the free energy has also been verified in some spin glass models at high temperature (replica-symmetric phase)  such as the Sherrington-Kirkpatrick model  \cite{ALR, comets1995sherrington} and the spherical Sherrington-Kirkpatrick model \cite{BaikLee} (and in \cite{landon} at critical temperature).


Our main result states that under the assumptions that the fourth moments of $w_e$ and $\nu_x$ are finite, $F$ obeys a Gaussian central limit theorem with an explicit rate of convergence,  as the size of the graph gets large but its maximum degree stays bounded. 
 Define the volume growth function of $G$ as 
 \begin{equation} \label{eq: volume_growth_function}\Psi_G(R) = \max_{ x \in V(G)} | V( \bB^x_R)|,  \end{equation}
where $\bB^x_R$ is the subgraph of $G$ induced by all the vertices that are at most  distance $R$  away from vertex $x$. Let $\Phi$ be the cumulative distribution function of the standard Gaussian distribution. 
\begin{thm}
	\label{thm: CLT_4_moments}
	Consider the disordered monomer-dimer model on a finite graph $G$ with maximum degree $D$. Assume that the vertex and the edge weights have finite fourth moments and the edge weights are non-degenerate (that is, the edge distribution is not a point mass). Then there exist constants $ c, C>0$, depending only on $D $ and the  distributions of the weights, 
	such that for all $R\geq 1$,
	\begin{dmath}\label{clt:error_bd}
	\sup_{s \in \R} \left|\prob\left(\frac{F - \E F}{\sqrt{\Var(F)}} \leq s \right) - \Phi(s)\right| \leq C \left (  \frac{  |V(G)|^{1/4}  }{ |E(G)|^{1/2}} \Psi_G(3R)^{1/4}  + \frac{  |V(G)|^{1/2}  }{ |E(G)|^{1/2}} e^{-cR} + \frac{  |V(G)|^{1/2}  }{\  |E(G)|^{3/4}} \right).
	\end{dmath}
	Assume further that for some constants $K_1$ and $K_2$, the weight distributions satisfy 
	\begin{itemize}
	\item[(i)] $\max( \E |w_e|^{2},  \E |\nu_x|^{2} ) \le K_1,$  
	 \item[(ii)]  $\E|w_e - w_e'| \ge K_2$ where $w_e'$ is an i.i.d.\ copy of $w_e$.
	\end{itemize}
	 Then the constant $C$ in \eqref{clt:error_bd} can be taken as  $C = C_1 (1+ \E |w_e|^4 + \E |\nu_x|^4)^{3/8}$  and the constants $c$ and $C_1$ can be chosen depending only on $D, K_1,$ and $K_2$.  
\end{thm}
Let $(G_n)_{n\geq 1}$ be a sequence of connected graphs with $|V(G_n)| \to \infty$ as $n \to \infty$. Assume that the maximum degrees of $G_n$ are bounded above by some constant $D$, which implies that $\Psi_{G_n}(R) \le CD^R$ for all $n, R \ge 1$. 
Also, the numbers of edges and vertices $G_n$ are of same order since $|V(G_n)| -1 \le |E(G_n)| \le (D/2) |V(G_n)|$. We can take $R = \sqrt{ \log |V(G_n)|}$ in \eqref{clt:error_bd}   to derive that the normalized $F_n$, the free energy associated with $\mu_{G_n}$,  converges in distribution to a standard Gaussian. To summarize, we have the following corollary.
\begin{cor}
    Let $(G_n)_{n\geq 1}$ be a sequence of finite graphs such that $|V(G_n)| \to \infty$, and such that the maximum degrees of $G_n$ are bounded above by some constant $D$. Assume that the vertex and the edge weights have finite fourth moment and the edge weights are non-degenerate. Writing $F_n$ for the free energy associated to $\mu_{G_n}$, one has
    \[
	\frac{F_n - \E F_n}{\sqrt{\Var(F_n)}}  \stackrel{d}{\to} N(0,1) \quad \text{as $n \to\infty$.}
	\]
    Here, $\stackrel{d}{\to}$ denotes the convergence in distribution.
\end{cor}

In fact, the fourth moment assumption can be weakened when the volume growth of $G_n$ is not too fast. To be more precise, if we further assume that the sequence of graphs $(G_n)_{n \ge 1} $ has uniformly sub-exponential volume growth, that is, for any $\alpha > 0$, there exists $K$ such that
 \[
	\Psi_{G_n} (R) \leq K \exp(\alpha R) \text{ for all } n, R \ge 1,
	\]
then we are able to prove a Gaussian central limit theorem  under a nearly-optimal finite $(2+\eps)$-moment assumption. 
\begin{thm}
	\label{thm: CLT_two_moments}
	Let $(G_n)_{n \geq 1}$ be a sequence of finite graphs  such that  $|V(G_n)| \to \infty$ as $n \to \infty$   and $|E(G_n)| \ge \delta |V(G_n)|$ for all $n$  large enough for some positive constant $\delta$. Furthermore, assume that the maximum degrees of $(G_n)_{n \ge 1} $ are bounded above by some constant $D$ and that  $(G_n)_{n \ge 1} $ has uniformly sub-exponential volume growth. Suppose that $\E |w_e|^{2+\eps} + \E |\nu_x|^{2+\eps} <\infty$ for some $\eps > 0$ and $w_e$ is non-degenerate.
	Then
	\[
	\frac{F_n - \E F_n}{\sqrt{\Var(F_n)}}  \stackrel{d}{\to} N(0,1) \quad \text{as $n \to\infty$,}
	\]
	where $F_n$ is the free energy associated with $\mu_{G_n}$.  
\end{thm}

\begin{rem}
The proof of Theorem~\ref{thm: CLT_two_moments} will show that the hypothesis of uniform subexponential volume growth can be replaced by the following weaker assumption. There
 exists $K$ such that
 \[
\Psi_{G_n} (R) \leq K \exp(\eps R/20) \text{ \ for all } n, R \ge 1.
\]
In the above,  $\eps$ is the same as one that appears in the moment condition $\E |w_e|^{2+\eps} + \E |\nu_x|^{2+\eps} <\infty$.
\end{rem}
The most important examples covered by the above theorem are the growing boxes of $\mathbb{Z}^d$ for any $d \ge 1$, which trivially satisfies the uniform subexponential volume growth assumption. 

In addition, we are also able to show an ``annealed'' central limit theorem for the number of dimers in the matching. Denote by $\la \cdot \ra$ the Gibbs expectation with respect to $\mu_G$, and let $\M_G$ be a random matching of $G$ sampled according to $\mu_G$.

For positive constants $\eps$ and $A$, define the following class of finite graphs
\[ \mathsf{Growth}(A, \eps) = \{ G: \Psi_G(R) \le A \exp(\eps R) \text{ for all } R \ge 1 \}.  \]

\begin{thm}
    \label{thm: no_of_dimers}
    Consider the disordered monomer-dimer model on a finite graph $G$ such that the vertex and the edge weights have finite fourth moments.  Assume that  maximum degree of $G$ is bounded above by $D$ and $G$ satisfies $|E(G)| \ge \delta | V(G)|$. Moreover, assume that there is  $c_0 > 0$ such that 
    \[ \Var(\la |\M_G| \ra) \geq c_0 |E(G)|.\]
     Then there exists $\eps >0$ that depends only on $D$ and the weight distributions such that for any $G \in \mathsf{Growth}(A, \eps)$ that also satisfies the above conditions the following holds. We can find constants $ c, C>0$ that depend only on $D, \delta, A$ and the weight distributions  such that  for all $R\geq 1$,
    \begin{dmath*}
	\sup_{s \in \R} \left|\prob\left(\frac{\la |\M_G| \ra - \E \la |\M_G| \ra}{\sqrt{\Var(\la |\M_G| \ra)}} \leq s \right) - \Phi(s)\right| \\
    \leq  Cc_0^{-1/2} \left( e^{-cR}  + |E(G)|^{-1/4}|\Psi_G(R)|^{2}\right) + C c_0^{-3/4}|E(G)|^{-1/4}. 
	\end{dmath*}
\end{thm}

It was shown in \cite[Corollary~5.10]{DeyKrishnan} using Heilmann-Lieb theory on the properties of the roots of the matching polynomials
that if the support of $w_e$ is contained in $[a,b]$ and the support of $\nu_x$ is contained in $[c,d]$, and if $b - 2c < - \log{D}$, then there exists $c_0 > 0$ (depending only on the weight distributions and $D$) such that $\Var(\la |\M_G| \ra) \geq c_0 |E(G)|$. Though the authors only considered the case that the graph $G$ is a cylinder graph,  the proof can be applied to general bounded degree graphs. 
Using a completely different technique, we provide a similar lower bound for Gaussian edge weights and general vertex weights on a general bounded degree graph.
\begin{prop}
\label{prop: var_lower_bound_dimers}
Assume that edge weights are i.i.d.\ standard Gaussians and the vertex weights are i.i.d.\ from any distribution. Then there exists 
a constant $c> 0$, depending only on $D$ and the vertex distribution, such that 
\[ \Var(\la |\M_G| \ra)  \ge c |E(G)|.\]
\end{prop}
We can apply Theorem~\ref{thm: no_of_dimers} and Proposition~\ref{prop: var_lower_bound_dimers}  to a sequence of bounded degree graphs $(G_n)_{n \ge 1}$ with uniformly sub-exponential volume growth when the edge weight distribution is Gaussian.
Note that uniformly sub-exponential volume growth implies that for any $\eps>0$, there exists a finite constant $A$ such that 
$G_n \in \mathsf{Growth}(A, \eps)$ for all $n \ge 1$. If we take, for example, $R = \sqrt{ \log |V(G_n)|}$, the bound of Theorem~\ref{thm: no_of_dimers} goes to zero as $n \to \infty$. Therefore, we obtain the following corollary. 
\begin{cor}\label{cor:clt_dimer}
    Let $(G_n)_{n \geq 1}$ be a sequence of finite graphs such that  $|V(G_n)| \to \infty$ as $n \to \infty$ and $|E(G_n)| \ge \delta |V(G_n)|$ for all $n$ large enough for some positive constant $\delta$. Assume that the maximum degrees of $(G_n)_{n \ge 1} $ are bounded above by some constant $D$ and that  $(G_n)_{n \ge 1} $ has uniformly sub-exponential volume growth. Furthermore, assume that the edge weights are i.i.d.\ standard Gaussian random variables and that vertex weight distribution has finite fourth moment. Then
    \[
        \frac{\la|\M_{G_n}| \ra_n - \E \la |\M_{G_n}| \ra_n}{\sqrt{\Var(\la |\M_{G_n}| \ra_n)}} \stackrel{d}{\to} N(0,1) \quad \text{as $n \to\infty$,}
    \]
    where $\la \cdot \ra_n$ is the Gibbs expectation with respect to $\mu_{G_n}$.
\end{cor}

To prove Theorem~\ref{thm: CLT_4_moments}, we apply a result by Chatterjee \cite{chatterjee2008new}, which has turned out to be quite successful in proving central limit theorems with a rate of convergence in a variety of problems in geometric probability, combinatorial optimization, and number theory \cite{chatterjee2008new, chatterjee2012random, chatterjee2017minimal, cao2021central}. 

Chatterjee's method quantifies the rough idea that a function $f(X_1, X_2, \ldots, X_N)$ of independent random variables $X_1, X_2, \ldots, X_N$ is asymptotically Gaussian if its partial derivatives $\partial_i f(X_1, X_2, \ldots, X_N)$ are `approximately independent'. See Theorem~\ref{lem: chatterjee} for a precise statement. For the free energy, these partial derivatives are given by  
\[ \partial_{w_e} F  = \left\la \mathbf{1}_{\{e\in \M_G\}}\right\ra,  \quad  \partial_{\nu_x} F  = \left\la \mathbf{1}_{\{x \not \in \M_G\}}\right\ra.   \]
To successfully apply Chatterjee's method, one needs to show that the above derivatives can be well-approximated by some functions that only depend on the randomness of the local neighborhood of $e$ or $x$.

 Let us restrict our discussion to the edge derivative $\partial_{w_e} F$ as the vertex derivative $\partial_{\nu_x} F$ behaves very similarly.  For an edge $e = (xy)$ of $G$, let $\bB^e_R$ denote the subgraph of $G$ induced by all the vertices that are at most distance $R$ away from either $x$ or $y$. Also, for a subgraph $H \subseteq G$, let $ \M_{H}$ denote a random matching on $H$ sampled according to the Gibbs measure restricted to the graph $H$ where the weights of the vertices and edges of $H$ are kept unchanged from $G$. The next result confirms that $\partial_{w_e} F$ can indeed be well-approximated by a local function. 
\begin{thm}	\label{thm: corr_decay_Lp}
Consider the disordered monomer-dimer model on a finite graph $G$  with maximum degree $D$.	
There exist constants $c, C>0$ depending on $D$ and the weight distributions such that the following holds. For any edge $e$ and for any $R\geq 1$, we have
\[	\E \big |\left\la \mathbf{1}_{\{e\in \M_G\}}\right\ra - \la \mathbf{1}_{\{e\in \M_{\bB^e_R}\}} \ra \big | \leq Ce^{-cR}.	\]\end{thm}
Theorem~\ref{thm: no_of_dimers} is also proved using Chatterjee's method. The partial derivative $\la |\M_G | \ra$ with respect to edge and vertex weights naturally involve edge-to-edge and vertex-to-vertex correlations (see \eqref{partial_e} and \eqref{partial_x}). To prove the locality of these derivatives, we establish an exponential two-point correlation decay in the disordered monomer-dimer model using a similar technique. For example, we show that
\begin{prop}
    \label{prop: corr_decay}
    Consider the disordered monomer-dimer model on a finite graph $G$  with maximum degree $D$.	
    There exist constants $c, C>0$ depending on $D$ and the weight distributions such that the following holds. For any edge $e, e'$, we have
\[ \E  |\left\la \mathbf{1}_{\{e\in \M_G, e' \in \M_G \}}\right\ra - \la \mathbf{1}_{\{e\in \M_{G}\}} \ra \la \mathbf{1}_{\{e'\in \M_{G}\}} \ra |  \leq Ce^{-c\mathrm{dist}(e, e')}.\]
\end{prop}

In  \cite{vdB}, van den Berg used a beautiful geometric argument involving disagreement percolation to control the effect of boundary in the standard monomer-dimer model (where $w_e \equiv \log \lambda$ and $\nu_x \equiv 0$ for some constant $\lambda >0$).
A similar argument can also be found in \cite{vdBB}. This showed exponential strong spatial mixing for the monomer-dimer model at any temperature and gave an alternative proof of the absence of phase transition. To prove Theorem~\ref{thm: corr_decay_Lp} and Proposition~\ref{prop: corr_decay}, we broadly follow van den Berg's argument. However, in the disordered case, we have random, possibly unbounded weights on the edge and vertices, which leads to additional complications.

Finally, let us mention that the bound on Chatterjee's theorem requires a finite fourth-moment assumption on the weight distributions. To remove this assumption, we carefully interpolate the {\em centered}  free energy with the original weights with the one with the appropriately truncated weights in the proof of Theorem~\ref{thm: CLT_two_moments}.

The rest of the paper is organized as follows. In Section~\ref{sec:corrdecay}, we prove the correlation decay result for the disordered monomer-dimer model. We then use it to prove the central limit theorems in Section~\ref{sec:CLT}. We conclude the paper by listing a few open problems in Section~\ref{sec:open_prob}.

\section{Correlation Decay}\label{sec:corrdecay}
This section contains the proof of  the correlation decay (Theorem~\ref{thm: corr_decay_Lp} and Proposition~\ref{prop: corr_decay}) in the disordered monomer-dimer model. In fact, a slightly stronger version of Theorem~\ref{thm: corr_decay_Lp} is required to prove Theorem~\ref{thm: CLT_4_moments}. If the weights are bounded, we can show exponential correlation decay for all possible realizations of the weights. For unbounded weights, the same conclusion holds only after we throw away a bad set of weights with an exponentially small probability.  

Call a function $\varphi : (0, \infty) \to (0, \infty)$  vanishing if   $\lim_{s \to \infty} \varphi(s)  = 0$. For a vanishing function $\varphi$, we say a random variable $X$ to be $\varphi$-bounded if 
\[ \prob( |X| \ge t ) \le \varphi(t) \text{ for all } t > 0. \]
For $U \subseteq V(G)$ and $F \subseteq E(G)$, let $\calF^{U \times F}$ denote the product $\sigma$-algebra generated by the independent random variables $\big( (\nu_x)_{x \in U}, (w_e)_{e \in F}\big).$ 
\begin{prop}
	\label{thm: corr_decay}
	Given a vanishing function $\varphi$ and integer $D \ge 1$, 
	there exist  constants $c, C>0$  depending only on $\varphi $ and $D$
	such that the following holds. Consider the disordered monomer-dimer model with $\varphi$-bounded edge and vertex weights on any  finite graph $G$ whose maximum degree is bounded above by $D$. Then for any $e = (xy) \in E(G)$ and for any $R\geq 1$,  there exists an event $A \in \calF^{ V(\bB^e_{R}) \setminus \{x, y\} \times  E(\bB^e_{R}) \setminus \{e\}} $ with $\prob(A) \geq 1 - C e^{-c R}$, such that  on the event $A$, we have 
	\[
	\left|\left\la \mathbf{1}_{\{e\in \M_G\}}\right\ra -  \left\la \mathbf{1}_{\{e\in \M_{H}\}}\right\ra \right| \leq Ce^{-cR},
	\]
	for any subgraph $H$ satisfying $\bB^e_R \subseteq H \subseteq G$.
\end{prop}
The above edge correlation decay result easily yields a similar correlation decay result for vertices. 
\begin{cor}
	\label{cor: corr_decay_vertex}
Assume the same set-up as in Proposition~\ref{thm: corr_decay}.
	Then for any for $x \in V(G)$ and any $R\geq 1$, 	 there exists an event $A \in \calF^{ V(\bB^x_{R}) \setminus \{x\} \times E(\bB^x_{R})} $  with $\prob(A) \geq 1 - C e^{-c R}$, such that on event $A$, we have 
	\[ \left|\left\la \mathbf{1}_{\{x \not \in \M_G\}}\right\ra -  \left\la \mathbf{1}_{\{x \not \in \M_{\bB^x_R}\}}\right\ra \right| \leq Ce^{-cR}.
	\]
\end{cor}
\begin{proof}[Proof of Corollary~\ref{cor: corr_decay_vertex}]
Let $e_1, e_2, \ldots, e_d $ be the edges incident on $x$. By assumption, $d \le D$. Clearly, for any matching of $G$  (or $\bB^e_R$), $x$ does not belong to the matching  if and only if  exactly one of the edges $e_1, e_2, \ldots, e_d$ belongs to that matching. Consequently, 
\[  \mathbf{1}_{\{x \not \in \M_G\}} -   \mathbf{1}_{\{x \not \in \M_{\bB^x_R}\}} =  \sum_{i=1}^d  \Big (  \mathbf{1}_{\{e_i  \in \M_G\}} -   \mathbf{1}_{\{e_i  \in \M_{\bB^x_R}\}} \Big).  \]
For the edge $e_i = (xx_i)$, we apply Proposition~\ref{thm: corr_decay} with $H =  \bB^x_R \supseteq \bB^{e_i}_{R-1}$ to obtain that
	\[
	\left|\Big \la \mathbf{1}_{\{e_i \in \M_G\}}\Big\ra -  \left\la \mathbf{1}_{\{e_i \in \M_{\bB^x_R}\}}\right\ra \right| \leq C_1 e^{-c_1(R-1)},
	\]
	on some event $A_i$ such that   $A_i \in \calF^{ V(\bB^{e_i}_{R-1}) \setminus \{x, x_i\} \times  E(\bB^{e_i}_{R-1}) \setminus \{e_i\}}  \subseteq \calF^{ V(\bB^x_{R}) \setminus \{x\} \times  E(\bB^x_{R}) }$ and  $\prob(A_i) \geq 1 - C_1 e^{-c_1 (R-1)}$. Moreover, the constants $c_1, C_1$ can be chosen depending only on $D$ and $\varphi$. Take $A = \cap_{i=1}^d A_i$ and set $c = c_1/2$ and $C = DC_1$. Note that $A$ remains  $ \calF^{ V(\bB^x_{R}) \setminus \{x\} \times E(\bB^x_{R})}$-measurable. 
	The corollary  now follows from the triangle inequality. 
\end{proof}
The following lemma is an immediate consequence of Proposition~\ref{thm: corr_decay} and Corollary~\ref{cor: corr_decay_vertex}.
\begin{lem}
	\label{lem: corr_decay_Lp}
	Assume the same set-up as in Proposition~\ref{thm: corr_decay}.
Then  for any edge $e$, any vertex $x$, and for any $R\geq 1$,  we have 
	\[
	\E \sup_{ w_e }  |\left\la \mathbf{1}_{\{e\in \M_G\}}\right\ra -  \la \mathbf{1}_{\{e\in \M_{\bB^e_R}\}} \ra |  \leq Ce^{-cR},
	\]
	\[
	\E \sup_{ \nu_x }  |\left\la \mathbf{1}_{\{x \not \in \M_G\}}\right\ra -  \la \mathbf{1}_{\{x \not \in \M_{\bB^x_R}\}} \ra |  \leq Ce^{-cR}, 
	\]
	where $c, C$ are constants that depend only on $D$ and $\varphi$.
\end{lem}
The rest of the section is devoted to proving Proposition~\ref{thm: corr_decay}.

\subsection{Markov random fields and path of disagreement}

We first recall a coupling idea by van den Berg  for a general Markov random field. Fix a finite graph $H$ and let $\Sigma = \{0,1\}^{V(H)}$ be  the space of binary spins indexed by its vertices. The elements of $\Sigma$ are called configurations. Let  $\lambda$  be a  probability measure on $\Sigma$ and  let $(\sigma_v)_{v \in V(H)}$ be a random element of $\Sigma$ drawn according to $\lambda$.  The probability measure $\lambda$  is called a Markov random field if for any $U \subseteq V(H)$, and for any configuration $\eta \in \{0,1\}^{U^c}$,
\[
\lambda\left(\sigma_v = \cdot,  \ v\in U \mid \sigma_v =\eta_v,  v\in U^c\right) = \lambda\left(\sigma_v = \cdot,  \ v\in U \mid \sigma_v = \eta_v, v\in \partial_V U\right),
\]
provided $\lambda ( \sigma_v = \eta_v, v\in \partial_V U) > 0$.
In above, for $U \subseteq V(H)$, $\partial_V U$ denotes  the outer vertex  boundary of $U$ in $H$, that is, $\partial_V U = \{ x \in V(H) \setminus U: (xy) \in E(H)  \text{ for some } y \in U  \}$.


Given $U \subseteq V(H)$ and a boundary condition $\eta \in \{0, 1\}^{\partial_V U}$, define the conditional measure of $\lambda$ on $\overline U : = U \cup \partial_V U $ by 
\[
\lambda_U^\eta (\sigma_v  = \cdot, v \in \overline U)  := \lambda(\sigma = \cdot \text{ on $U$}\mid \sigma = \eta \text{ on $\partial_V U$}),
	\]
provided $\lambda ( \sigma = \eta \text{ on $\partial_V U$} ) >0$, in which case we say $\eta$ to be admissible. 

\begin{lem}[\cite{vdB}, Lemma 11.2.1]
	\label{thm: couple}
	Let $\lambda$ be a Markov random field on $H$. Let $ v \in U \subseteq V(H)$ and let $\eta^1,\eta^2 \in \{0,1\}^{\partial_V U }$ be two admissible boundary conditions. Let $\Pi$  be the product measure $ \lambda_U^{\eta^1} \otimes \lambda_U^{\eta^2} $  on $\{0, 1\}^{\overline U} \times \{0, 1\}^{\overline U}$.	Then
	\[ \big|   \lambda_U^{\eta^1}( \sigma_v = 1) -   \lambda_U^{\eta^2}( \sigma_v = 1)  \big |  \le \Pi \Big ( \exists \textnormal{ a path of disagreement from $v$ to $\partial_V U$}\Big). \]
We say a pair  of spin configurations $(\alpha^{1}_u, \alpha^{2}_u)_{u \in \overline U }$ has a path of disagreement from $v$ to $\partial_V U$  if there exists a finite path of adjacent vertices $v_0 = v, v_1, v_2, \ldots, v_\ell$  in $\overline U$ such that $v_\ell \in \partial_V U$ and $\alpha^{1}_{v_i} \ne \alpha^{2}_{v_i}$ for all $1 \le i \le \ell$. 
\end{lem}

We are going to apply the above result to the monomer-dimer model. A matching $ M \in \mathcal{M}_G$ can be naturally viewed as a configuration $\alpha \in \{0, 1\}^{E(G)}$ where $\alpha_e = 1$ if and only if the edge $e$ belongs to $M$. 

To use the setting of Lemma~\ref{thm: couple}, the monomer-dimer model can be recast as a probability measure on the spins indexed by vertices of the line graph of $G$, denoted by $\widehat G$, as follows.
	\begin{equation}
		\label{eq: factorize}
		\widehat \mu( \alpha) \propto \prod_{e \in  V(\widehat G) } e^{ w_e   \alpha_e}\prod_{x\in V(G)}  \left(\mathbf{1}_{\{\sum_{e: e\ni x} \alpha_e = 1\}} + e^{\nu_x} \mathbf{1}_{\{\sum_{e: e\ni x}  \alpha_e = 0\}}\right), \quad \alpha \in \{0, 1\}^{V(\widehat G)}.
	\end{equation}
The measure $\widehat \mu$  admits a clique factorization in $\widehat{G}$, since, for each $x \in V(G)$, the set of all edges of $G$ incident on $x$ forms a clique in $\widehat G$. Hence,  $ \widehat \mu$, for a fixed realization of weights,  defines a Markov random field with respect to $\widehat{G}$.

 For $F \subseteq E(G)$, by the  boundary of $F$,  we mean its outer edge boundary in $G$,  which is defined as $\partial_E F = \{ e \in E(G) \setminus F: \exists e' \in F, e \sim e'\}$, where $e \sim e'$ means that the edges $e$ and $e'$ are adjacent in $G$. Also, let $\overline F =  F \cup \partial_E F$. In the context of matching, a boundary condition $\eta \in \{0, 1\}^{\partial_E F}$ is  admissible if $\{e \in \partial_E F:  \eta_e  = 1\}$ is a matching in $G$. In such case, let $ \M^\eta_F$ denote the random matching of $(V(G), \overline F)$  drawn from the conditional Gibbs measure $\widehat \mu^\eta_{F}$. 
Lemma~\ref{thm: couple},  specialized to the monomer-dimer model, now reads as follows. 
\begin{lem}\label{lem:coupling}  
Consider a disordered monomer-dimer model on a finite graph $G$ with a fixed realization of the weights.  Let $e \in F \subseteq E(G)$  and  let $\eta^1, \eta^2$ be  two admissible boundary conditions. Suppose that the matchings $ \M^{\eta^1}_F$ and $\M^{\eta^2}_F$ are drawn independently from $\widehat \mu^{\eta^1}_{F}$ and $\widehat \mu^{\eta^2}_{F}$ respectively and set $\M_1 = \M^{\eta^1}_F$ and $ \M_2 =  \M^{\eta^2}_F$ for brevity. 
Let $\Pi$  be the distribution $(\M_1, \M_2)$.  Then 
	\[ \big|   \la \mathbf{1}_{ \{ e \in \M_1 \} }   \ra -  \la \mathbf{1}_{ \{ e \in \M_2 \} } \ra   \big |  \le \Pi \big ( \exists \textnormal{ a path of disagreement from $e$ to $\partial_E F$}\big). \]
A path of disagreement from $e$ to $\partial_E F$ means a finite sequence of edges $e_0 = e \sim e_1 \sim \cdots \sim e_\ell$ such that $e_\ell \in \partial_E F$ and 
$e_i \in \M_1 \oplus \M_2,$ the symmetric difference of $\M_1$ and $\M_2$, for each $1 \le i \le \ell$.
\end{lem}
For matching, disagreement paths have a special structure as they can never intersect. Indeed, for a pair of matchings $M_1$ and $M_2$ of $G$, $M_1 \oplus M_2$ is the union of disjoint self-avoiding paths (or cycles).  If not, then we can find three distinct edges in  $M_1 \oplus M_2$ who share a common vertex. But then at least two of these edges must belong to either $M_1$ or $M_2$, which is impossible. As a result, any disagreement path is (vertex) self-avoiding and  
 there can be at most two disagreement paths from $e$ to $\partial_EF$, originating at either endpoint of $e$. Moreover, the consecutive edges on any disagreement path must belong alternatively  to $M_1$ and $M_2$.

\subsection{Proof of Proposition~\ref{thm: corr_decay}}

Fix $e = (xy) \in E(G), R \ge 1$ and let $H$ be a subgraph containing $\bB^e_R$. Let $\partial H$ be the outer edge-boundary of the edge set $E(H)$ in $G$. That is, $\partial H = \partial_E E(H)$.

If $\partial \bB^e_{R} =  \emptyset$, then $\bB^e_R = H = G$ and the proposition holds trivially as the difference between the two Gibbs expectations is  identically zero. So, we assume that  $\partial \bB^e_{R} \ne \emptyset$.  For notational simplicity, let us write $\M_{H}^\eta$ for $\M_{E(H)}^\eta$, where $\eta \in \{0, 1\}^{\partial H}$ is an admissible boundary condition. 
 
By Markov property, we have 
\[ \left \la \mathbf{1}_{\{e\in \M_G\}}\right\ra  = \sum_\eta \big  \la \mathbf{1}_{\{e\in \M_{H}^\eta\}} \big \ra \mu_{\partial H}(\eta),  \]
where $\mu_{\partial H}$ is the marginal of $\mu_G$ on  $\partial H$ and  the sum is over all possible admissible boundary conditions $\eta$. It suffices to show that 
there exists an event $A \in \calF^{ V(\bB^e_{R}) \setminus \{x, y\} \times  E(\bB^e_{R}) \setminus \{e\}},$ independent of the choice of $\eta$, with $\prob(A) \geq 1 - C e^{-c R}$, such that on the event $A$, 	\[
	\max_\eta \big |\big\la \mathbf{1}_{\{e\in \M_H^\eta \}}\big\ra -  \big\la \mathbf{1}_{\{e\in \M_{H}^0\}}\big\ra \big | \leq Ce^{-cR},\]
and the constants $c$ and $C$ can be chosen depending only on $D$ and $\varphi$.
To achieve this, we  fix an admissible boundary condition $\eta$ on $\partial H$. Let $\Pi$ be the independent coupling of $\big(\M^\eta_H, \M^{0}_H \big)$. By Lemma~\ref{lem:coupling}, we deduce that 
\begin{align}
	 \big | \big \la \mathbf{1}_{\{e\in \M_H^\eta \}}\big\ra -  \big \la \mathbf{1}_{\{e\in \M_{H}^0\}}\big\ra \big | &\le \Pi \big ( \exists \text{ a path of disagreement w.r.t.\ $(\M^\eta_H, \M^{0}_H)$  from $e$ to $\partial H$}\big) \nonumber \\
	&\le 	\Pi \big ( \exists \text{ a path of disagreement w.r.t.\ $(\M^\eta_H, \M^{0}_H)$  from $e$ to $\partial \bB^e_R$}\big). \label{eq:disagree}
\end{align}

	Let $a >0$ be sufficiently large. We call an edge $f$ of  $\bB^e_R$ to be `good' if (i) $f$ is at least distance $3$ away from $e$ and $\partial \bB^e_R$ and (ii) 
	the weights on $f$ and all the edges adjacent  to $f$  are bounded above by $a$, and the weights on the end-vertices of these edges are bounded below by $-a$.  Let  $c_1 > 0$ be a constant  to be chosen later and let $A$ be the event that all self-avoiding paths from $e$ to $\partial \bB_e^R$ contain at least $c_1R$ many good edges. From the definition of good edges, it is clear that $A$  is measurable with respect to $\calF^{ V(\bB^e_R) \setminus \{x, y\} \times  E(\bB^e_R) \setminus \{e\}}.$ 
	
	We claim that for $a>0$ sufficiently large and for $c_1>0$ sufficiently small, there exist constants $c, C> 0$, depending only on $D$ and $\varphi$, such that
	\begin{equation}\label{eq:exception_set_prob}
	\prob(A) \ge 1  -  Ce^{-cR}.
	\end{equation}
To prove this claim, let us fix a self-avoiding path from $e$ to $\partial \bB^e_{R}$ of  length of $\ell \ge R$. By a greedy search in the subgraph $\bB^e_{R-3} \setminus \bB^e_{3}$, one can  find a deterministic set of edges $S$ of size at least $\ell/C_1$ on that path, where $C_1 \ge 1$ depends only on $D$, such that the $2$-neighborhoods of the edges in $S$ are pairwise disjoint and do not intersect with $e$ and $\partial \bB^e_R$. Note that the events $\{ f \text{ is good} \}_{f \in S}$ are independent and  for any $\delta>0$, there exists $a_0> 0$, depending only on $D$ and $\varphi$, such that $\prob( f \text{ is good} ) \ge 1 - \delta$ for all $f \in S$ and $a \ge a_0$. From Chernoff bound, for any constant $C_3 > 0$, we can choose $\delta>0 $ sufficiently small which can be guaranteed by choosing $a$ sufficiently large,  such that 
	\begin{equation}\label{eq:good_edge1}
	\prob\left(\text{at least half of the edges in $S$ are good}\right) \geq 1 - C_2 e^{- C_3 |S|} \ge 1 - C_2 e^{- (C_3/ C_1) \ell} ,
	\end{equation}
	for some absolute constant $C_2$. Choose $C_3$  so that \eqref{eq:good_edge1} is at least $1 - C_2 D^{-3\ell}$ and then
 take $c_1 = 1/(2C_1)$ in the definition of the event $A$. Since there are at most $2D^\ell$  self-avoiding paths from $e$ of length $\ell$, by a union bound, we have 
\[ \prob(A^c) \le \sum_{\ell  = R}^\infty  2 D^\ell  \cdot C_2 D^{-3\ell} \le C e^{-cR}, \]
for appropriate choices of $c$ and $C$.

Next, we claim that if  $A$ holds, then  there exist at least $c_1 R$  edge-disjoint cut-sets consisting of good edges which separate $e$ from $\partial \bB_R^e$.  To argue that let us introduce the first-passage distance between any two vertices $u, v$ as the minimum number of good edges on a (self-avoiding) path between $u$ and $v$. Now for $k \ge 0,$ let $\mathcal{B}_k$  be the set of vertices with the first-passage distance at most $k$ from  either of the endpoints of $e$. Clearly, $\mathcal{B}_{k} \subseteq \mathcal{B}_{k+1}$ and  on the event $A$, we have $\mathcal{B}_{c_1 R -1} \subseteq V(\bB^e_R).$ For each $1 \le k \le c_1R$,  let $\mathcal{C}_k$ be set of the boundary edges of $\mathcal{B}_{k-1}$, i.e., all edges of the form $(uv)$ with $u \in \mathcal{B}_{k-1}$ and $v \not \in \mathcal{B}_{k-1}$. Obviously, $\mathcal{C}_{k}$ is a cut-set since any path escaping from $\mathcal{B}_{k-1}$ to outside must use one of the edges in $\mathcal{C}_k$. Moreover, all edges of $\mathcal{C}_{k}$ must be good. To see this, fix an edge $(uv) \in \mathcal{C}_{k}$ with $u \in \mathcal{B}_{k-1}$ and $v \not \in \mathcal{B}_{k-1}$. If this edge is not good, we can reach $v$ from $e$ using only $k-1$ good edges
contradicting the fact  $v \not \in \mathcal{B}_{k-1}$.

We now fix a realization of the weights such that $A$ holds and proceed to bound the probability in \eqref{eq:disagree}. As mentioned before, there can be at most two paths of disagreement from $e = (xy)$ to $\partial \bB_R^e$. For definiteness, let us fix one of them. Denote it by $\gamma$, where  $\gamma(0) = e \sim \gamma(1) = e_1 \sim \gamma(2) = e_2 \sim \cdots $ is the enumeration of adjacent edges of $\gamma$ starting from~$e$.  Define the random times $\tau_0 = 0 \le \tau_1 < \tau_2 < \cdots $ such that $\tau_k = \inf \{ i : \gamma(i) \in \mathcal{C}_k \} \in \mathbb{Z}_+ \cup \{ \infty\}$. Clearly, to reach  $\partial \bB^e_R$, the path $\gamma$ needs to cross all the cutsets $\mathcal{C}_1,  \mathcal{C}_2, \ldots, \mathcal{C}_{c_1 R}$ at least once, which implies that $\tau_k < \infty $ for each $1 \le k \le c_1 R$. Therefore,  
\[   \Pi ( \text{$\gamma$ reaches $\partial \bB^e_R$}  )  \le \prod_{k=0}^{c_1 R  - 1} \Pi \big(  \tau_{k+1} < \infty  \mid \tau_k < \infty \big). \]
To bound $\Pi \big(  \tau_{k+1} < \infty  \mid \tau_k < \infty \big)$,  we condition on a finite value of $\tau_k$, the edges  $\gamma(i), 0 \le  i \le \tau_k$ such that $\gamma$ enters $\mathcal{C}_k$ at time $\tau_k$ for the first time,  and on the event whether $e \in \M_H^\eta$ or  $e \in \M_H^0$ (which, in turn, determines whether $\gamma(i) \in  \M_H^\eta$ or  $\gamma(i) \in \M_H^0$ for each $i \ge 1$ by the alternating property of the path of disagreement) and we seek to bound the conditional  probability that $\gamma$ can be extended one step further after $\tau_k$. 

For definiteness, suppose that $e_{\tau_k} \in  \M_H^\eta \setminus  \M_H^0$. Let $z$ be the end vertex of $e_{\tau_k}$, which is not shared by $e_{\tau_k -1}$. Let $f_1 = (z z_1) ,\ldots, f_d = (z z_d)$ be the set of those edges incident to $z$ which do not share a common vertex with any of the edges  $\gamma(i), 0 \le  i \le \tau_k-1$. In particular, $e_{\tau_k}$ is excluded. 

 If  $d = 0$,  then $\gamma$ can not be extended and the conditional probability is zero. So, assume that $1 \le d \le D$. Since $e_{\tau_k}$ is a good edge, each $f_i$ lies inside $\bB^e_R$. To be able to extend $\gamma$ one step further,  one of the edges $f_1, f_2, \ldots, f_d$ must belong to $\M_H^0$ and none of them must belong to $\M_H^\eta$. We further condition on the information whether $ f \in \M_H^0$ and $ f \in \M_H^\eta$ for every edge $f \in E(H) \setminus \{ f_1, f_2, \ldots, f_d \}$ that are not on $\gamma(i), 0 \le  i \le \tau_k$ (we have already conditioned on  them before). Let $Q \subseteq \{1, 2, \ldots, d\}$ be the set of the indices of the vertices among $z_1, z_2, \ldots, z_d$ which remain unmatched in $ \M_H^0 \setminus \{f_1, f_2, \ldots, f_d\}$ after the conditioning. Since $\M_H^0$ and $\M_H^\eta$ are independent, the conditional probability in question is bounded above by the probability that one of these edges $f_1, f_2, \ldots, f_d$ belongs to $\M_H^0$ after we condition on  the information whether $f \in \M^0_H$ for all edges $f$  of $H$ except $f_1, f_2, \ldots, f_d$. This is given by 
\begin{equation}\label{eq:star_prob}
	\frac{\sum_{i \in Q} \exp(w_{f_i} + \sum_{j \in Q \setminus \{i\} } \nu_{z_j})}{\exp(\nu_z + \sum_{j \in Q} \nu_{z_j}) + \sum_{i \in Q} \exp(w_{f_i} + \sum_{j \in Q \setminus \{i\} } \nu_{z_j})} =  \frac{\sum_{i \in Q} \exp(w_{f_i}  - \nu_{z_i} - \nu_z)}{1 + \sum_{i \in Q} \exp(w_{f_i}  -  \nu_{z_i} - \nu_z)}.
	\end{equation}
	Since $e_{\tau_k}$ is good, the edge and vertex weights of $f_1, \ldots, f_d$ are all bounded above and below by $a$ and $-a$ respectively, which implies that \eqref{eq:star_prob} is bounded above by $De^{3a}(1+De^{3a})^{-1}$. Hence, 
	\[ \Pi \big(  \tau_{k+1} < \infty  \mid \tau_k < \infty \big) \le \frac{De^{3a}}{1+De^{3a}},\]
	and consequently, on the event $A$, 
	\[   \Pi ( \text{$\gamma$ reaches $\partial \bB_R^e$}  ) \le \left ( \frac{De^{3a}}{1+De^{3a}} \right )^{c_1 R } \le Ce^{-cR},\]
	for some constants $c, C>0$ which can be chosen depending only on $D$ and $\varphi$, as promised. This completes the proof of the proposition.  \qed

  \subsection{Proof of Proposition~\ref{prop: corr_decay}}
In this section, we sketch a proof of the following bounds on
edge-to-edge and vertex-to-vertex correlations. In particular, we will prove Proposition~\ref{prop: corr_decay} follows.  As the arguments are quite similar to the previous section,  we will omit some details.
\begin{prop}\label{prop:ee_vv_local}
    Consider the disordered monomer-dimer model on a finite graph $G$  with maximum degree $D$.	
    There exist constants $c, C>0$ depending on $D$ and the weight distributions such that the following holds. Let $R \ge 1$ and $e, e' \in E(G)$ and $x, x' \in V(G)$. 
Suppose that $B_R(e, e')$ and $B_R(x, x')$ are arbitrary subgraphs of $G$ containing $\bB_R^e \cup \bB_R^{e'}$  and  $\bB_R^x \cup \bB_R^{x'}$ respectively. Then for any $\delta_1, \delta_2 \in \{0, 1\}$, we have 
\begin{align}
    \E \sup_{w_e, w_{e'}} \left |\la \mathbf{1}_{\{e \in \M_G  \}} = \delta_1,  \mathbf{1}_{\{e' \in \M_G  \}} = \delta_2 \ra -  \la \mathbf{1}_{\{e \in \M_{B_R(e, e')}  \}} = \delta_1,  \mathbf{1}_{\{e' \in \M_{B_R(e, e')}  \}} = \delta_2 \ra  \right | &\leq Ce^{-cR}, \label{eq:cor_tv_2edge}\\
        \E \sup_{\nu_x, \nu_{x'}} \left |\la \mathbf{1}_{\{x \not \in \M_G  \}} = \delta_1,  \mathbf{1}_{\{x' \not \in \M_G  \}} = \delta_2 \ra -  \la \mathbf{1}_{\{x \not  \in \M_{B_R(x, x')}  \}} = \delta_1,  \mathbf{1}_{\{x' \in \M_{B_R(x, x')}  \}} = \delta_2 \ra  \right | &\leq Ce^{-cR}. \label{eq:cor_tv_2vertex}
\end{align}
\end{prop}

\begin{proof}[Proof of Proposition~\ref{prop:ee_vv_local}]
    First, we note that Lemma~\ref{thm: couple} holds with greater generality. Indeed Lemma~11.2.1 of \cite{vdB} states that instead of considering the one-dimensional marginals of the measures $\lambda^{\eta_1}_U, \lambda^{\eta_2}_U$, one can compare their finite-dimensional marginals whose total variational distance can be bounded above by the $\Pi$-probability of having a path of disagreement from that finite set of vertices to the boundary $\partial_V U$. Applying it to two-dimensional marginal on the edges $e, e'$,  we can obtain the following analogous statement of Lemma~\ref{lem:coupling}: using the same notations as in the lemma, for $e,e'\in F$, one has
\[\begin{gathered}
\max_{\delta_1, \delta_2 \in \{0, 1\}} \big|   \la \mathbf{1}_{ \{ e \in \M_1 \} } = \delta_1,  \mathbf{1}_{ \{ e' \in \M_1 \} } = \delta_2   \ra -  
  \la \mathbf{1}_{ \{ e \in \M_2 \} } = \delta_1,  \mathbf{1}_{ \{ e' \in \M_2 \} } = \delta_2   \ra  \big | \\
 \le \Pi \big ( \exists \textnormal{ a path of disagreement from $\{e, e'\}$ to $\partial_E F$}\big).   
\end{gathered}
\]
This implies that, following the same arguments in Proposition~\ref{thm: corr_decay}, for any $e=(xy)$ and $e'=(x'y')$ and for any $R\geq 1$, there exists an event 
$A \in \mathcal{F}^{V(H)} \setminus \{x, y, x', y'\} \times E(V(H)) \setminus\{e, e'\}$ with $\prob(A) \geq 1 - Ce^{-cR}$ such that on the event $A$ one has
\begin{equation} \label{eq:bdd_edge_edge_prob_bd}
  \left |\la \mathbf{1}_{\{e \in \M_G  \}} = \delta_1,  \mathbf{1}_{\{e' \in \M_G  \}} = \delta_2 \ra -  \la \mathbf{1}_{\{e \in \M_{B_R(e, e')}  \}} = \delta_1,  \mathbf{1}_{\{e' \in \M_{B_R(e, e')}  \}} = \delta_2 \ra  \right | \leq Ce^{-cR}  
\end{equation}
for any subgraph $H$ satisfying $\bB^e_R \cup \bB^{e'}_R \subseteq H \subseteq G$. This implies \eqref{eq:cor_tv_2edge} after taking $H = B_R(e, e')$.

Also,  \eqref{eq:cor_tv_2vertex} also follows from \eqref{eq:bdd_edge_edge_prob_bd}
and from the fact that if $e_1, e_2, \ldots, e_d$ are the edges incident on $x$ then
\[
\mathbf{1}_{\{x\not\in \M_G\}} = \sum_{i=1}^d \mathbf{1}_{\{e_i \in \M_G\}}, \quad \mathbf{1}_{\{x\not\in \M_{B_R(x, x')}\}} = \sum_{i=1}^d \mathbf{1}_{\{e_i \in \M_{B_R(e, e')}\}}
\]
using arguments quite similar to Corollary~\ref{cor: corr_decay_vertex}. We omit the details.
\end{proof}
\begin{cor}\label{prop: corr_decay_stronger}
    Consider the disordered monomer-dimer model on a finite graph $G$  with maximum degree $D$.	
    There exist constants $c, C>0$ depending on $D$ and the weight distributions such that the following holds. For any edges $e, e'$ and vertices $x, x'$, we have
\begin{align} 
\label{eq: edge_edge_corr}
     \E \sup_{w_e, w_{e'}} |\left\la \mathbf{1}_{\{e\in \M_G, e' \in \M_G \}}\right\ra - \la \mathbf{1}_{\{e\in \M_{G}\}} \ra \la \mathbf{1}_{\{e'\in \M_{G}\}} \ra | &\leq Ce^{-c\mathrm{dist}(e, e')},\\
    \label{eq: edge_vertex_corr}
   \E \sup_{\nu_x, \nu_{x'}} |\left\la \mathbf{1}_{\{x \not \in \M_G, x' \not \in \M_G \}}\right\ra - \la \mathbf{1}_{\{x \not \in \M_{G}\}} \ra \la \mathbf{1}_{\{x' \not \in \M_{G}\}} \ra | &\leq Ce^{-c\mathrm{dist}(x, x')}
  \end{align}
\end{cor}
Note that \eqref{eq: edge_edge_corr} immediately implies Proposition~\ref{thm: corr_decay}.
 \begin{proof}
 We only prove \eqref{eq: edge_edge_corr} as the proof of \eqref{eq: edge_vertex_corr} is quite similar. 
     Let $\mathrm{dist}(e, e') \ge 2R +2$. Then \eqref{eq:cor_tv_2edge}, after taking $B_R(e, e')$ to be the disjoint union of $\bB^e_R$  and $\bB^{e'}_R$, implies that
     \begin{align}\label{eq:101}
         \E \sup_{w_e, w_{e'}} |\left\la \mathbf{1}_{\{e, e' \in \M_G\}}\right\ra -  \la \mathbf{1}_{\{e, e' \in \M_{B_R(e, e')}\}} \ra  | \le Ce^{-cR}.
     \end{align}
     On the other hand, by Markov property,  we have
\begin{equation}\label{eq:102}
\la \mathbf{1}_{\{e, e ' \in \M_{B_R(e, e')} \}} \ra = \la \mathbf{1}_{\{e\in \M_{\bB^e_R}\}} \ra \la \mathbf{1}_{\{ e'\in \M_{\bB^{e'}_R}\}} \ra.
\end{equation}
 Now taking $\delta_1  =1$ and summing over $\delta_2$ in \eqref{eq:cor_tv_2edge} and using triangle inequality, 
 \[  \E \sup_{w_e, w_{e'}}  | \la \mathbf{1}_{\{ e \in  \M_G\} } \ra -  \la \mathbf{1}_{\{e\in \M_{B_R(e, e')}\}}  \ra|
  = \E \sup_{w_e, w_{e'}}  | \la \mathbf{1}_{\{ e \in  \M_G\} } \ra -  \la \mathbf{1}_{\{e\in \M_{\bB^e_R}\}}  \ra|  \le C e^{-cR}. \]
  The same holds for $e'$. Therefore, we obtain that 
\begin{equation}\label{eq:103}
\E \sup_{w_e, w_{e'}} |  \la \mathbf{1}_{\{e\in \M_{G}\}} \ra \la \mathbf{1}_{\{e'\in \M_{G}\}} \ra -  \la \mathbf{1}_{\{e\in \M_{\bB^e_R}\}} \ra \la \mathbf{1}_{\{ e'\in \M_{\bB^{e'}_R}\}} \ra | \le Ce^{-cR}.
\end{equation}
Now \eqref{eq:101}, \eqref{eq:102} and \eqref{eq:103} together imply \eqref{eq: edge_vertex_corr}. 
 \end{proof}

\section{Proofs of the central limit theorems} \label{sec:CLT}
In this section, we prove Theorem~\ref{thm: CLT_4_moments} and Theorem~\ref{thm: CLT_two_moments}.  Throughout the section,  $c, c_0$ and $C, C_0, C_1, C_2$ will be positive constants that depend only on $D, K_1$, and $K_2$, unless mentioned otherwise. However, the values of $c$ and $C$ may vary from line to line. 
\subsection{Proof of Theorem~\ref{thm: CLT_4_moments}}
The proof of the central limit theorem will be based on a result of Chatterjee \cite{chatterjee2008new}, which we state below. Let $X = (X_1, X_2, \ldots, X_N)$ be a vector of independent random variables and let $g: \mathbb{R}^N \to \R$ be a measurable function. Suppose that $X' = (X_1', X_2', \ldots, X_N')$ is an independent copy of $X$. For any subset $S \subseteq \{1, 2, \ldots, N\}$, define the random vector $X^S$ as
\begin{align*}
X^S_i =  \left \{ \begin{array}{cc}   X_i' & \text{ if } i \in S, \\
 X_i & \text{ if } i  \not \in S.
\end{array} \right.
\end{align*}
For each $i$ and $S \subseteq \{1, 2, \ldots, N\}$ with $ i \not \in S$, set
\begin{equation}
\label{eq: Delta_g}
    \Delta_i g = g(X) - g(X^{\{i\}}), \ \ \   \Delta_i g^S = g(X^S) -g(X^{S \cup \{i\}}).   \end{equation}
 Finally, let $\sigma^2 = \Var(g(X))$.
\begin{thm}[Corollary~3.2 of \cite{chatterjee}]
	\label{lem: chatterjee}
	For $i, j \in \{1, 2, \ldots, N \}$, let $c(i, j)$ be a constant such that for all $S \subseteq \{1, 2, \ldots, N\} \setminus\{i\}$ and $T \subseteq \{1, 2, \ldots, N\} \setminus\{j\}$, one has
	\[
	\Cov(\Delta_i g \Delta_i g^S, \Delta_{j} g \Delta_{j} g^T) \leq c(i, j).
	\]
	Then
	\[
	\sup_{s \in \R} \left|\prob\left(\frac{g(X) - \E g(X) }{\sigma} \leq s\right) - \Phi(s)\right| \leq \frac{\sqrt{2}}{\sigma}\left(\sum_{i, j =1}^N c(i, j)\right)^{1/4} + \frac{1}{\sigma^{3/2}} \left( \sum_{i=1}^N \E |\Delta_i g|^3\right)^{1/2}.
	\]
\end{thm}
We  apply the above theorem to the free energy $F = \log{Z}$, viewed as a function of independent random variables $(\nu_x)_{x \in V(G)}$ and $(w_e)_{e \in E(G)}$. Let $I = V(G) \cup E(G)$ be the union of the set of vertices and edges of $G$, which serves as the common index set of all  random weights. Clearly, $|V(G)| \le |I| \le (1 +D/2) |V(G)|$.
As before, let $(\nu'_x)_{x \in V(G)}$ and $(w'_e)_{e \in E(G)}$ denote the independent resamples of the vertex and edge weights. 
 The discrete derivatives $\Delta_i F$ are given as follows. 
\begin{equation}
	\label{eq: diff_free_en1}
\Delta_e F = \int_{w'_e}^{w_e} \left\la \mathbf{1}_{\{e \in \M_G \}} \right\ra_{\mid w_e  = t}~\text{d}t, \quad \text{ if $i = e$ is an edge}, 
\end{equation}
\begin{equation}
	\label{eq: diff_free_en2}
\Delta_x F = \int_{\nu'_x }^{\nu_x} \left\la \mathbf{1}_{\{ x \not \in  \M_G \}} \right\ra_{\mid \nu_x  = t}~\text{d}t, \quad  \text{ if $i = x$ is a vertex.}
\end{equation}
For $S \subseteq I \setminus \{i\}$, we have a similar expression for $\Delta_i F^S$, but  the Gibbs expectation is now taken after the weights with indices belonging to $S$ are resampled. 

Let $F_{[i, R]}$ be the free energy associated with the Gibbs measure $\mu_{ \bB^e_R} $ or $\mu_{ \bB^x_R} $ depending on whether $i =e$, an edge or $i = x$, a vertex. 
For $i \not \in S$, we approximate the discrete derivative $\Delta_i F^S$ by the discrete derivative $\Delta_i F_{ [i, R]}^S$ of the local function  $F_{[i, R]}.$ 
We apply \eqref{eq: diff_free_en1}  separately for $G$ and $\bB_R^e$  and use  Lemma~\ref{lem: corr_decay_Lp} to  obtain that
\begin{align}
 \E \left |   \Delta_e F^S -   \Delta_e F^S_{[e, R]}   \right |^4  &\le \E |w_e  - w'_e|^4   \E \sup_{ t }  \Big |\left\la \mathbf{1}_{\{e\in \M_G\}}\right\ra_{\mid w_e  = t} -  \la \mathbf{1}_{\{e\in \M_{\bB^e_R}\}} \ra_{\mid w_e  = t} \Big |^4 \nonumber\\
 &\le C_0 \E |w_e|^4 e^{-c R}. \label{eq:EdecayR1}
 \end{align}
 Similarly, we have
 \begin{align}
 \E \left |   \Delta_x F^S -   \Delta_x F^S_{[x, R]}   \right |^4   &\le C_0 \E |\nu_x|^4 e^{-c R}.  \label{eq:EdecayR2}
 \end{align}
 Note that if we assume that $\E|w_e|^2, \E|\nu_x|^2 \le K_1$, then the random weights are all $\varphi$-bounded with $\varphi(t)  = K_1 t^{-2}$. Therefore, Lemma~\ref{lem: corr_decay_Lp} ensures that the constants $c$ and $C_0$ in  \eqref{eq:EdecayR1} and \eqref{eq:EdecayR2} can be chosen depending only on $D$ and $K_1$. 
Also, \eqref{eq: diff_free_en1} and  \eqref{eq: diff_free_en2} yield  the following trivial bound
\[  |\Delta_e F^S| \leq |w_e - w_e'| \ \  \text{ and } \ \  |\Delta_x F^S| \le   |\nu_x - \nu_x'|. \]
Consequently, $\E |\Delta_i F^S| ^4 \le 16(  \E |w_e|^4 + \E |\nu_x|^4)$. The same argument also implies that  $\E |\Delta_i F_{[i, R]}^S| ^4 \le 16(  \E |w_e|^4 + \E |\nu_x|^4)$.

To bound $c(i, j)$, let us define the error terms
\[  E_{[i, R]} = \Delta_i F  - \Delta_i F_{[i, R]}, \quad E^S_{[i, R]} = \Delta_i F^S  - \Delta_i F^S_{[i, R]} \ \  \text{ for } i \not \in S.
\]
We can then write, for $i \not \in S$ and $j \not \in T$, 
\begin{align*}
	&\Cov(\Delta_i  F\Delta_i F^S, \Delta_{j} F \Delta_{j} F^{T})\\
	=&\;\Cov \Big(  ( \Delta_i  F_{[i, R]} + E_{[i, R]})(  \Delta_i  F^S_{[i, R]} + E^S_{[i, R]}), ( \Delta_j  F_{[j, R]} + E_{[j, R]})( \Delta_j  F^T_{[j, R]} + E^T_{[j, R]}) \Big).
\end{align*}
By the bilinearity of the covariance, the above expression can be expanded as the sum of $16$ terms. The first term is given by
\[ \Cov(  \Delta_i  F_{[i, R]}  \Delta_i  F^S_{[i, R]},  \Delta_j  F_{[j, R]}  \Delta_j  F^T_{[j, R]} ).\] 
The key observation is that  $F_{[i, R]}$ and $F^S_{[i, R]}$ (and hence $\Delta_i  F_{[i, R]}$ and  $\Delta_i  F^S_{[i, R]}$ as well) only depend on the weights (either original or resampled)  of the vertices and edges that are distance at most $R$ from $i$. 
As a result,  the above covariance vanishes if $\mathrm{dist}(i, j) >2R+1$. Meanwhile, if $\mathrm{dist}(i, j) \le 2R+1$, then  it can be bounded above by $32(  \E |w_e|^4 + \E |\nu_x|^4)$.

Each of the remaining $15$ terms involves at least one error term and they all  can be bounded  by a similar approach. 
For example, let us consider a term of the form $\Cov(E_{[i, R]}W, YZ)$ where $W, Y, Z$ are appropriate discrete derivatives. By H\"older's inequality,  
\begin{align*}
|\Cov(E_{[i, R]}W, YZ)| &\leq \E |E_{[i, R]}W YZ| + \E|E_{[i, R]}W| \E|YZ|\\
&\leq   \big(   \E | E_{[i, R]} |^4  \E W^4 \E Y^4 \E Z^4  \big) ^{1/4}  + \big(   \E |E_{[i, R]} |^2  \E W^2  \big) ^{1/2} \big (\E Y^2 \E Z^2  \big) ^{1/2} \\
&\leq  2 \big(   \E|  E_{[i, R]} |^4  \E W^4 \E Y^4 \E Z^4  \big) ^{1/4}.
\end{align*}
The fourth moments of $W, Y,$  and $Z$ are bounded above by  $16 (  \E |w_e|^4 + \E |\nu_x|^4)$.  On the other hand,  $ \E | E_{[i, R]} |^4 \le C_0 (  \E |w_e|^4 + \E |\nu_x|^4) e^{-cR}$ by  \eqref{eq:EdecayR1} and  \eqref{eq:EdecayR2}. Consequently, 
\[ |\Cov(E_{[i, R]}W, YZ)| \le 32 C_0 (  \E |w_e|^4 + \E |\nu_x|^4) e^{-cR}. \]
%
Summing up the $16$ covariance terms,  we can take, as long as $\mathrm{dist}(i, j)>2R+1$,
\[
c(i, j) \leq C_1(  \E |w_e|^4 + \E |\nu_x|^4) e^{-cR}.
\]
If $\mathrm{dist}(i, j)\leq 2R+1$, we can just take  $c(i, j) = C_1(  \E |w_e|^4 + \E |\nu_x|^4)$.  Let us point out that  the constant $C_1$ only depends on $D$ and $K_1$.

Note that for any  $i$, the number of $j \in I$ such that $\mathrm{dist}(i, j)\leq 2R+1$ is bounded above by $10D\Psi_G(3R)$.
In summary, we obtain 
\begin{align}
	\sum_{i, j \in I } c(i, j ) &\leq  C_1(  \E |w_e|^4 + \E |\nu_x|^4) \cdot \left (  \left  | \Big \{ ( i, j) \in I \times I :  \mathrm{dist}(i, j)\leq 2R+1 \Big \} \right |   +  e^{-cR} \cdot | I|^2 \right) \nonumber \\
	&\le  C_1 (1+ D/2)^2 (  \E |w_e|^4 + \E |\nu_x|^4) \Big(  |V(G)| \Psi_G(3R)   +     |V(G)|^2 e^{-cR} \Big).
	    \label{eq: corr_terms_bound}
	\end{align}
	Next, we estimate 
	\begin{align*}
	\sum_{i \in I} \E | \Delta_i F|^3 &\le  \sum_{i \in I} (\E | \Delta_i F|^4)^{3/4} \\
	 &\le C_1(  \E |w_e|^4 + \E |\nu_x|^4)^{3/4}|V(G)|.
	\end{align*}
	To apply Theorem~\ref{lem: chatterjee}, it remains to find a suitable lower bound on the variance of $F$. In the following lemma, we will show that $\sigma^2: = \Var(F) \ge c_0 |E(G)|$.   Combining these estimates,  we conclude that 
	\begin{multline*} 
	   \frac{\sqrt{2}}{\sigma}\Big (\sum_{i, j \in I } c(i, j)\Big)^{1/4} + \frac{1}{\sigma^{3/2}} \Big( \sum_{i \in I } \E |\Delta_i F |^3\Big)^{1/2} \\ \le 
	  C_2 (  \E |w_e|^4 + \E |\nu_x|^4)^{1/4}  \left ( \frac{  |V(G)|^{1/4}  }{ |E(G)|^{1/2}} \Psi_G(3R)^{1/4}  +  \frac{  |V(G)|^{1/2}  }{ |E(G)|^{1/2}} e^{-cR} \right) + C_2(  \E |w_e|^4 + \E |\nu_x|^4)^{3/8} \frac{  |V(G)|^{1/2}  }{ |E(G)|^{3/4}}, 
	\end{multline*} 
	which completes the proof of Theorem~\ref{thm: CLT_4_moments}. \qed

\begin{lem}
	\label{lem: var_lower_bound}  
	(a) We have 
	\[  \Var(F) \le 2 ( \E |w_e|^2 |E(G)|+  \E |\nu_x|^2 |V(G)|). \]
	(b) Assume that there exist positive constants $K_1$ and $K_2$ such that 
	$\max(\E|w_e|^{2}, \E |\nu_x |^{2})  \le K_1 $ and $\E|w_e - w'_e| \ge K_2$ where $w_e'$ is an i.i.d.\ copy of $w_e$.
	Then there exists a constant $c_0 > 0$ depending only on $D,  K_1,$ and $K_2$ such that 
	\[
	\Var(F) \geq c_0 |E(G)|.
	\]
\end{lem}
\begin{proof}
	The upper bound is an easy consequence of Efron-Stein inequality, see \cite[Lemma~3.3]{DeyKrishnan}. 
	 For the lower bound, we will follow the ideas of \cite[Lemma~3.6]{DeyKrishnan}. However, we need to be careful in keeping track of  the dependence of the constant $c_0$ on the distribution of the weights. 
	 Let $m = |E(G)|$. We enumerate the edges in $E(G)$ as $\{1, 2, \ldots, m\}$, and their respective weights are denoted as $\{w_1,\ldots, w_m\}$. Consider the edge-revealing martingale $(\E[ F \mid \calF_j])_{0 \le j \le m}$, where 
	$\mathcal{F}_j = \sigma(w_1,\ldots, w_j)$. The variance of $F$ can then be written as
	\[ \Var(F) =  \sum_{j=1}^m  \E \Big ( \E[F \mid \mathcal{F}_j] - \E[ F \mid \mathcal{F}_{j-1}] \Big )^2. \]
	Let $F^{(j)}$ be obtained from $F$ after replacing $w_j$ with an i.i.d.~copy $w_j'$.  Then the above martingale difference is  equal to 
	\[  \E[F \mid \mathcal{F}_j] - \E[ F \mid \mathcal{F}_{j-1}]  =  \E[F  - F^{(j)}\mid \mathcal{F}_{j}].  \]
	Therefore, we write 
	\begin{align}
	\Var(F) &= \sum_{j=1}^m \E\left[\E[ F - F^{(j)} \mid \mathcal{F}_j]^2\right]  \nonumber\\
	&\geq \sum_{j=1}^m \E\left[\E[F - F^{(j)} \mid w_j]^2\right]  \nonumber\\
&=  \frac{1}{2} \sum_{j=1}^m \E_{w_j, w_j'} \Big [   \E' (F  - F^{(j)})^2 \Big],  \label{eq: var_lower_bound2}
\end{align}
where $\E'$ is the conditional expectation given $w_j$ and $w_j'$. 
For $e = (xy)$, let  $\alpha_e$ and $\beta_e$ denote the partition functions associated with the subgraphs of $G$ obtained by removing the edge $e$ and the vertices $x, y$ respectively. Clearly, $\alpha_e$ and $\beta_e$ does not depend on $w_e$. We have
$Z_G = \alpha_{e_j} + \beta_{e_j}\exp(w_j)$, and
	\[
	F^{(j)} - F = \int_{w_j}^{w_j'} \frac{\beta_{e_j} e^t}{\alpha_{e_j} + \beta_{e_j} e^t}~\text{d}t.
	\]
	By monotonicity of $t \mapsto \beta_{e_j} e^t (\alpha_{e_j} + \beta_{e_j} e^t )^{-1}$, we obtain
	\[
	| \E' ( F^{(j)} - F)  | \ge |w_j'  -  w_j| \cdot \E' \frac{\beta_{e_j} e^{\min(w_j, w_j')} }{\alpha_{e_j} + \beta_{e_j} e^{\min(w_j, w_j')} }
	\]
	
	Let $e_j = (x_jy_j)$ and  $E_j$ be the set of all edges adjacent to $e_j$. Observe that 
	\[
	\alpha_{e_j} \leq e^{\nu_{x_j} + \nu_{y_j}} \cdot \beta_{e_j} \prod_{e = (uv)\in E_j} (1 + e^{w_e - \nu_u - \nu_v}).
	\]
	From the above inequality, we deduce that
	\begin{align*}
	| \E' ( F^{(j)} - F)  |  &\ge  |w_j'  -  w_j| \E' \left(1 + e^{\nu_{x_j} + \nu_{y_j} - \min(w_j, w_j')} \prod_{e = (uv)\in E_j} (1+e^{w_e - \nu_u - \nu_v})\right)^{-1} \\
	&\ge |w_j'  -  w_j|  \mathbf{1}_{ \{ \min (w_j, w_j')  \ge  - a \} } \cdot  \E \left(1 + e^{\nu_{x_j} + \nu_{y_j} + a} \prod_{e = (uv)\in E_j} (1+e^{w_e - \nu_u - \nu_v})\right)^{-1},
	\end{align*}
	for any  $a > 0$. Note that
	\begin{equation}\label{eq:CS_lower}
	\E  |w_j'  -  w_j|^2   \mathbf{1}_{\{  \min (w_j, w_j')  \ge  - a \} }  \ge \Big (  \E  |w_j'  -  w_j|   \mathbf{1}_{\{  \min (w_j, w_j')  \ge  - a \} } \Big)^2 \\
\end{equation}
and 
	\begin{align*}
	\E  |w_j'  -  w_j|   \mathbf{1}_{\{  \min (w_j, w_j')  \ge  - a \} }  & \ge \E  |w_j'  -  w_j|    -  \E  |w_j'  -  w_j|   \mathbf{1}_{\{  \max (|w_j|, |w_j'|) >    a \} } \\
	&\ge  K_2  -  \big (  \E  |w_j'  -  w_j|^{2} \big)  \prob (  \max (|w_j|, |w_j'|) >    a) ^{1/2}\\
	&\ge K_2  -  4 K_1 ( 2 K_1 / a^{2} ) ^{1/2} \ge K_2/2,
	\end{align*}
	if $a = a_0$ is chosen sufficiently large depending on $K_1$ and $K_2$.  On the other hand, the random variable $\Big(1 + e^{\nu_{x_j} + \nu_{y_j} + a_0} \prod_{e = (uv)\in E_j} (1+e^{w_e - \nu_u - \nu_v})\Big)^{-1}$ is independent of $w_j$ and $w_j'$ and its expectation can be bounded below by a positive constant $c_1$ that depends only on $D, a_0,$ and $K_1$, but not on $j$ or $m$. So finally, by  \eqref{eq:CS_lower}, 
	\begin{align*}
	\E_{w_j, w_j'} \Big [  \big|  \E' (F  - F^{(j)})  \big|^2 \Big]     \ge c_1^2    \E  |w_j'  -  w_j|^2   \mathbf{1}_{\{  \min (w_j, w_j')  \ge  - a_0 \} } \ge c_1^2 (K_2/2)^2.
	\end{align*}
	The desired lower bound on $\Var(F)$ now follows from   \eqref{eq: var_lower_bound2}.
\end{proof}

\subsection{Proof of Theorem~\ref{thm: CLT_two_moments}}
Without loss of generality, assume that $|V(G_n)| = n$. By hypothesis, we have  $|E(G_n)| \ge \delta n$. Let $L  = n^{\kappa}$ for some small positive constant $\kappa > 0$ ($\kappa = 1/10$ suffices). For $t\in [0,1]$, define the weights 
	\begin{align*}
	w_e^t &= w_e\mathbf{1}_{\{|w_e|\leq L\}} + tw_e\mathbf{1}_{\{|w_e| > L\}},\\
	\nu_x^t &= \nu_x \mathbf{1}_{\{|\nu_x|\leq L\}} + t\nu_x \mathbf{1}_{\{|\nu_x |> L\}},
	\end{align*}
	which interpolate between the weights truncated at $L$ at $t=0$ and the original weights at $t=1$. 
	Let $F_{n,t}$ be the free energy of the disordered monomer-dimer model on $G_n$ with  the weights $(w_e^t)_{e \in E(G_n) } , (\nu_x^t)_{x \in V(G_n)} $, and write $\la\cdot\ra_{n, t}$ for the corresponding  Gibbs expectation. For a random variable $W$, we write $\overline{W} = W - \E W$ for its centered version. 
	
	It follows from Theorem~\ref{thm: CLT_4_moments} that 
	\begin{equation}\label{eq:base_CLT}
	 \frac{\overline{F_{n, 0}}}{\sqrt{\Var(F_{n,0})} } \stackrel{d}{\to} N(0, 1).
	 \end{equation}
	Indeed, by assumption,
	$\E |w^0_e|^{2}$ and  $\E |\nu^0_x|^{2}$  are bounded above by  $K_1 := \max(\E |w_e|^{2}, \E |\nu_x|^{2} ) ) <\infty$.  For large $n$,  
	\[ \E| w^0_e  - (w_e^0)' | \ge K_2 := \frac{1}{2} \E |w_e  - w_e'| > 0.\]
	 Moreover, we have the trivial bound $\E |w^0_e|^4, \E |\nu^0_x|^4 \le n^{4 \kappa}$. Therefore, Theorem~\ref{thm: CLT_4_moments} yields that 
\begin{equation} \label{eq:2nd_moment_CLT_base}
 \sup_{s \in \R} \left|\prob\left(\frac{\overline{F_{n, 0} }}{\sqrt{\Var(F_{n, 0} )}} \leq s \right) - \Phi(s)\right| \leq C n^{3\kappa/2}   \left ( n^{-1/4} \Psi_{G_n}(3R)^{1/4}  +  e^{-cR} + n^{-1/4} \right)
\end{equation}
	for any $R \ge 1$, where $c$ and $C$ depend on $\delta, K_1, K_2,$ and $D$ only. For $R = 2c^{-1} \log n$, the uniformly sub-exponential volume growth   implies that there exists $K$ such that \[ \Psi_{G_n}( 3R) \le K \exp( (1/10) \log n).  \]
After plugging in, the RHS of \eqref{eq:2nd_moment_CLT_base} becomes
	\[ C n^{3\kappa/2}   \left ( n^{-1/4} \exp( (1/40) \log n)  +  e^{- 2\log n} + n^{-1/4} \right)
\to 0 \quad \text{as } n \to \infty, \]
and \eqref{eq:base_CLT} follows. 
From Lemma~\ref{lem: var_lower_bound},  the variance of the free energy is of order $n$, namely, 
\begin{equation}\label{eq:var_bd}
cn \le \Var(F_{n,0}) \le Cn, \quad cn \le \Var(F_{n,1}) \le Cn.
\end{equation}
Our main claim is as follows.
\begin{equation}\label{eq:F_0F_1_close}
n^{-1} \E |\overline{F_{n,0}} - \overline{F_{n,1}}|^2 \to 0.
\end{equation}
Once the claim is established, it follows from the variance upper bounds in \eqref{eq:var_bd} and from the Cauchy-Schwarz inequality that 
\begin{align*}
n^{-1} | \Var(F_{n, 0}) - \Var(F_{n, 1}) | &  \le  n^{-1} \big(\Var(F_{n,0}) + \Var(F_{n,1}) \big)^{1/2}     \big(\E |\overline{F_{n,0}} - \overline{F_{n,1}}|^2 \big)^{1/2}      \\
&\le C  \big ( n^{-1} \E |\overline{F_{n,0}} - \overline{F_{n,1}}|^2 \big )^{1/2}  \to 0.
\end{align*}
Coupled with the lower bound on the variance provided in \eqref{eq:var_bd}, the above  implies that 
\begin{equation}\label{eq:var_equivalence}
 \frac{ \Var(F_{n, 0})}{  \Var(F_{n, 1})} \to 1.
 \end{equation}
Finally, we write
\[  \frac{\overline{F_{n, 1}}}{\sqrt{\Var(F_{n, 1})} } = \frac{\overline{F_{n, 0}}}{\sqrt{\Var(F_{n,0})} } \cdot \sqrt{\frac{ \Var(F_{n, 0})}{  \Var(F_{n, 1})}} + 
\frac{\overline{F_{n, 1}} - \overline{F_{n, 0}}}{\sqrt{n}} \cdot \sqrt{\frac{n}{\Var(F_{n, 1})}} \]
and conclude that 	
\[  \frac{\overline{F_{n, 1}}}{\sqrt{\Var(F_{n, 1})} } \stackrel{d}{\to} N(0, 1)\]
from \eqref{eq:base_CLT}, \eqref{eq:var_equivalence}, \eqref{eq:F_0F_1_close},  and \eqref{eq:var_bd}. 

It remains to show the claim \eqref{eq:F_0F_1_close}. Note that 
	\begin{align*}
	\overline{F_{n,1}} - \overline{F_{n,0}} &= \int_0^1 \frac{\text{d}}{\text{d}t} \overline{F_{n, t}}~\text{d}t \\
	&= \int_0^1  \sum_{e\in E(G_n)}   \overline{\la  \mathbf{1}_{\{e \in \M_{G_n} \}}\ra_{n, t} w_e \mathbf{1}_{\{|w_e|>L\}}}~\text{d}t + \int_0^1  \sum_{x\in V(G_n)} \overline{ \la \mathbf{1}_{\{x \not \in \M_{G_n} \}}\ra_{n, t} \nu_x \mathbf{1}_{\{|\nu_x|>L\}}}~\text{d}t.
	\end{align*}
	For notational simplicity, write
\begin{align*}
X^e_{n, t} &=  \sum_{e\in E(G_n)}   \overline{\la  \mathbf{1}_{\{e \in \M_{G_n} \}}\ra_{n, t} w_e \mathbf{1}_{\{|w_e|>L\}}}, \quad   X^x_{n, t} =  \sum_{x\in V(G_n)} \overline{ \la \mathbf{1}_{\{x \not \in \M_{G_n} \}}\ra_{n, t} \nu_x \mathbf{1}_{\{|\nu_x|>L\}}}.
\end{align*}
By Jensen,
\begin{align*} 	
 \E |\overline{F_{n,0}} - \overline{F_{n,1}}|^2 \le 2\int_0^1 \E | X^e_{n, t} |^2~\text{d}t + 2\int_0^1 \E | X^x_{n, t} |^2~\text{d}t. 
\end{align*} 	
To tackle the first integral above, let $\bB^e_{n, R}$  be the ball of radius $R$ around $e$ in $G_n$.  Define
	\[
	\gamma_{n, t}(e, R) = \la \mathbf{1}_{\{e\in \M_{ \bB^e_{n, R} } \} }\ra_{n, t}, \quad \text{ and } \ \ \  Y^e_{n, t} =  \sum_{e\in E(G_n)}   \overline{\gamma_{n, t}(e, R)  w_e \mathbf{1}_{\{|w_e|>L\}}}.	\]
 By Lemma~\ref{lem: corr_decay_Lp},  there exist constants $C, c>0$ (depending only on $D$ and $K_1$) such that for all $t\in [0,1]$,
\[	 \E \sup_{w_e^t} \big | \la \mathbf{1}_{\{e\in \M_{G_n}\}} \ra_{n, t} - \gamma_{n, t}(e, R) \big|^2 \leq Ce^{-cR} = Cn^{-2},\]
which implies that
\begin{align*}	\E \Big |  \overline{\la  \mathbf{1}_{\{e \in \M_{G_n} \}}\ra_{n, t} w_e \mathbf{1}_{\{|w_e|>L\}}} &- \overline{\gamma_{n, t}(e, R)  w_e \mathbf{1}_{\{|w_e|>L\}}} \Big|^2  \le  \E \Big | \big (  \la \mathbf{1}_{\{e\in \M_{G_n}\}} \ra_{n, t} - \gamma_{n, t}(e, R) \big)  w_e \Big|^2 \\
&\leq   \E \sup_{w_e^t} \big |  \la \mathbf{1}_{\{e\in \M_{G_n}\}} \ra_{n, t} - \gamma_{n, t}(e, R) \big|^2 \cdot  \E |w_e|^2\\
 &\le C  n^{-2}.
\end{align*}
Therefore, by Cauchy-Schwarz, 
	\begin{equation}\label{eq:1}
\sup_{ t \in [0, 1]} \E \big | X^e_{n, t}-  Y^e_{n, t} \big|^2 \leq  C.
\end{equation}
	
Note that
\[  \E |Y^e_{n, t}|^2 = \sum_{e, e' \in E(G_n) } \Cov \big( \gamma_{n, t}(e, R)  w_e \mathbf{1}_{\{|w_e|>L\}}, \gamma_{n, t}(e', R)  w_{e'} \mathbf{1}_{\{|w_{e'}|>L\}} \big). \]
If the distance between $e$ and $e'$ is greater than $2R+1$, $\gamma_{n, t}(e, R)$ and $\gamma_{n, t}(e', R)$ are independent and 
the covariance vanishes. Else, by Cauchy-Schwarz and H\"older's  inequalities, the absolute value of the covariance is bounded above by
\[ \E [w_e^2 \mathbf{1}_{\{|w_e|>L\}}]  \le  \frac{\E |w_e|^{2+\eps}}{L^\eps} \le C n^{-\kappa \eps}.  \]
Due to uniform subexponential volume growth assumption on $G_n$,  for any $\alpha>0$, the number of pairs $(e, e')$ with distances at most $2R+1$ is bounded above by $Cn e^{ \alpha \log n }$ for $C$ sufficiently large. Consequently, after choosing $\alpha = \kappa \eps/2$, we have 
\[ \sup_{ t \in [0, 1]} n^{-1} \E |Y^e_{n, t}|^2 \le n^{-1} \cdot  Cn e^{ \alpha \log n } \cdot C n^{-\kappa \eps} \to 0.\]
This, together with \eqref{eq:1}, implies that 
\begin{equation} \label{eq:2}
 \sup_{ t \in [0, 1]} n^{-1} \E |X^e_{n, t}|^2 \to 0.
 \end{equation}
An exactly similar argument shows that 
 \begin{equation} \label{eq:3}
 \sup_{ t \in [0, 1]} n^{-1} \E |X^x_{n, t}|^2 \to 0.
 \end{equation}
The claim \eqref{eq:F_0F_1_close} now follows from \eqref{eq:2} and \eqref{eq:3}  and  the proof of the theorem is now complete. \qed

\subsection{Proof of Theorem~\ref{thm: no_of_dimers}}
In this section, we prove Theorem~\ref{thm: no_of_dimers}. Again we will use Theorem~\ref{lem: chatterjee}. This time we will apply it to $\la |\M_G|\ra$, viewed as a function of the edge and vertex weights. When we compute the derivatives of $\la |\M_G|\ra$ with respect to the weights, the Gibbs covariances of the indicators  $\mathbf{1}_{\{e\in\M_G\}}$ and $\mathbf{1}_{\{x\not\in\M_G\}}$  with $\la |\M_G|\ra$ appear naturally, and we use Corollary~\ref{prop: corr_decay_stronger} and Proposition~\ref{prop:ee_vv_local}  to bound the covariances of these derivatives.


Set $\Lambda = \la |\M_G|\ra$. It is straightforward to check that
\begin{align}
    \partial_{w_e} \Lambda &= \la |\M_G| \mathbf{1}_{\{e \in \M_G\}}\ra - \la |\M_G| \ra \la \mathbf{1}_{\{e \in \M_G\}}\ra, \label{partial_e}\\
        \partial_{\nu_x} \Lambda &= \la |\M_G| \mathbf{1}_{\{x \not \in \M_G\}}\ra - \la |\M_G| \ra \la \mathbf{1}_{\{x \not \in \M_G\}}\ra
        =   -\frac{1}{2} \left ( \la |\mathsf{U}_G| \mathbf{1}_{\{x \not \in \M_G\}}\ra - \la   |\mathsf{U}_G| \ra \la \mathbf{1}_{\{x \not \in \M_G\}}\ra \right), \label{partial_x}
\end{align}
where $\mathsf{U}_G$  is the set of all  vertices of $G$ left unmatched by $ \M_G$. The last equality from above follows from the simple identity that $|\mathsf{U}_G| + 2|\M_G| = |V(G)|$. Using the linearity of the Gibbs covariance, we can then  express the discrete derivatives of $\Lambda$ as
\begin{align*}
\Delta_e \Lambda &= \sum_{f \in E(G)} \int_{w_e'}^{w_e} (\la  \mathbf{1}_{\{e, f \in \M_G\}}\ra_{e, t} - \la \mathbf{1}_{\{e \in \M_G\}}\ra_{e, t}\la \mathbf{1}_{\{f \in \M_G\}}\ra_{e, t})~\mathrm{d}t,\\
\Delta_x \Lambda &= -\frac{1}{2} \sum_{y \in V(G)} \int_{\nu_x'}^{\nu_x} (\la  \mathbf{1}_{\{x, y \not \in \M_G\}}\ra_{x, t} - \la  \mathbf{1}_{\{x \not \in \M_G\}}\ra_{x, t} \la  \mathbf{1}_{\{y \not \in \M_G\}}\ra_{x, t})~\mathrm{d}t,
\end{align*}
where, for notational simplicity, we write $\la\cdot\ra_{e, t} = \la \cdot \ra_{|w_e = t}$ and  $\la\cdot\ra_{x, t} = \la \cdot \ra_{|\nu_x = t}$.

We first consider the $\Delta_e \Lambda$ term. 
Denote the correlation between edges $e$ and $f$,
\[ X_{f, e} = \sup_t |\la\mathbf{1}_{\{e, f\in \M_G\}}\ra_{e, t} - \la \mathbf{1}_{\{e\in \M_G\}}\ra_{e,t} \la \mathbf{1}_{\{f \in \M_G\}}\ra_{e, t}|. \]
By Jensen's inequality, for $p \ge 1$, 
\begin{align*}
   \| \Delta_e \Lambda \|_p \leq \sum_{f \in E(G)} \| w_e - w_{e'}\|_p \|X_{f, e}\|_p \le 2^p \| w_e \|_p \sum_{f \in E(G)}  \|X_{f, e}\|_p.
\end{align*}
By Corollary~\ref{prop: corr_decay_stronger}, one has, for $p \ge 1$,
\begin{equation*}\label{eq:p_norm_bound}
\| X_{f, e}\|_p \le C_* e^{-c_* \cdot \mathrm{dist}(e, f)}.
\end{equation*}
Then for $G \in \mathsf{Growth}(A, c_*/5)$,
\begin{align*}
    \sum_{f \in E(G) }  e^{-c_* \cdot \mathrm{dist}(e, f) } &\le 
    \sum_{k=0}^\infty \sum_{f:\mathrm{dist}(e, f) = k}  e^{-c_* k} \\
      &\le  \sum_{k=0}^\infty D\Psi_G(k)   e^{-c_* k} \\
    &\le  D A \sum_{k=0}^\infty   e^{(c_*/5) k}   e^{-c_* k} \\
    &\le C'.
\end{align*}
Therefore, we obtain that 
\begin{equation}\label{eq:Delta_e_bd}
 \| \Delta_e \Lambda \|_p \le C \| w_e\|_p. 
\end{equation}
Given  $e \in E, R \ge 1$, we define a local version of the edge discrete derivative $\Delta_e \Lambda$ as follows:
\[
\Delta_e \Lambda_{[e, R]} := \sum_{f\in E(\bB_{R}^e)} \int_{w_e'}^{w_e} \la\mathbf{1}_{\{f \in \M_{\bB_{2R}^e}\}} \mathbf{1}_{\{e \in \M_{\bB_{2R}^e}\}}\ra_{e, t} - \la\mathbf{1}_{\{f \in \M_{\bB_{2R}^e}\}}\ra_{e, t}\la\mathbf{1}_{\{e \in \M_{\bB_{2R}^e}\}}\ra_{e, t}~\mathrm{d}t.
\]

To control $\| \Delta_e \Lambda  - \Delta_e \Lambda_{[e, R]} \|_p$, we decompose 
\[
\sum_{f \in E(G)} \la\mathbf{1}_{\{e, f\in \M_G\}}\ra_{e, t} - \la \mathbf{1}_{\{e\in \M_G\}}\ra_{e,t} \la \mathbf{1}_{\{f \in \M_G|\}}\ra_{e, t} = Y_{R, e, t} + L_{R, e, t} + E_{R, e, t},
\]
where
\begin{align*}
    Y_{R, e, t} &= \sum_{f \in E(G) \setminus E(\bB_{R}^e)}  \la\mathbf{1}_{\{e, f\in \M_G\}}\ra_{e, t} - \la \mathbf{1}_{\{e\in \M_G\}}\ra_{e,t} \la \mathbf{1}_{\{f \in \M_G\}}\ra_{e, t},\\
    L_{R, e, t} &= \sum_{f\in E(\bB_{R}^e)} \la \mathbf{1}_{\{e, f \in \M_{\bB_{2R}^e}\}}\ra_{e, t} - \la\mathbf{1}_{\{e \in \M_{\bB_{2R}^e}\}}\ra_{e, t}\la\mathbf{1}_{\{f \in \M_{\bB_{2R}^e}\}}\ra_{e, t},\\
    E_{R, e, t} &= \sum_{f \in E(\bB_{R}^e)} \big( \la\mathbf{1}_{\{e, f \in \M_G\}}\ra_{e, t} -  \la\mathbf{1}_{\{e, f \in \M_{\bB^e_{2R}}\}}\ra_{e, t} \big) \\
    &- \sum_{f \in E(\bB_{R}^e)} \big( \la \mathbf{1}_{\{e\in \M_G\}}\ra_{e,t} \la \mathbf{1}_{\{f \in \M_G\}}\ra_{e, t} -  \la \mathbf{1}_{\{e\in \M_{\bB_{2R}^e} \}}\ra_{e,t} \la \mathbf{1}_{\{f \in \M_{\bB_{2R}^e} \}}\ra_{e, t} \big).
\end{align*}

We bound the three terms separately. 
 By Corollary~\ref{prop: corr_decay_stronger} and  under the volume growth assumption
 $G \in \mathsf{Growth}(A, c_*/5)$, 
\begin{align}
   \| \sup_t |Y_{R, e, t} | \|_p &\le \sum_{f \in E(G) \setminus E(\bB_{R}^e)} \| X_{f, e}\|_p  \nonumber \\
    &\le \sum_{ k > R} \sum_{f: \mathrm{dist}(e, f) = k} C_* e^{-c_* k} \nonumber \\
    &\le C' e^{-(4c_*/5)R}. \label{eq: X_bound}
    \end{align}
Next, for any $f, e$ with $\mathrm{dist}(e, f) \le R$, an application of Proposition~\ref{prop:ee_vv_local} with $B_R(e, f) = \bB_{2R}^e$ yields
\begin{align*}
\| \sup_t | \la\mathbf{1}_{\{e, f \in \M_G\}}\ra_{e, t} -  \la\mathbf{1}_{\{e, f \in \M_{\bB^e_{2R}}\}}\ra_{e, t}| \|_p &\le C_* e^{-c_* R} \\
\| \sup_t | \la\mathbf{1}_{\{f \in \M_G\}}\ra_{e, t} -  \la\mathbf{1}_{\{f \in \M_{\bB^e_{2R}}\}}\ra_{e, t}| \|_p &\le C_* e^{-c_* R}
\end{align*}
Hence,  it follows from the assumption $G \in \mathsf{Growth}(A, c_*/5)$ that,
\begin{equation}
\label{eq: E_bound}
    \| \sup_t |E_{R, e, t}|\|_p \leq  100C_* e^{-c_* R} |E(\bB_R^e)| \leq C'e^{-(4c_*/5)R}.
\end{equation}
Finally, we have the trivial bound
\begin{equation}
\label{eq: L_bound}
\| \sup_t |L_{R, e, t }|\|_p \leq |E(\bB_{R}^e)| \leq C'\Psi_G(R).
\end{equation}
Therefore, we deduce from \eqref{eq: L_bound} that
\begin{equation} \label{eq:local_L_p_bd}
    \| \Delta_e \Lambda_{[e, R]}\|_p \le  \|  w_e -  w_{e'} \|_p \| \sup_t |L_{e, R, t}| \|_p \le C \| w_e\|_p \Psi_G(R).
\end{equation}
 On the other hand, by \eqref{eq: X_bound} and \eqref{eq: E_bound}, we have
\begin{align}
\label{eq: compare_discrete_derivatives}
\| \Delta_e \Lambda - \Delta_e \Lambda_{[e, R]} \|_p \leq \|w_e - w_e'\|_p  (\| \sup_t |Y_{R, e, t}\|_p  + \| \sup_t E_{R, e, t}| \|_p ) \leq C \| w_e\|_p  e^{-(4c_*/5)R}.
\end{align}
By H\"older's inequality and using the bound
$\| X -\E X\|_4 \le 2 \| X\|_4$, we obtain that for any random variables $X, X', Y, W, Z$,
\[ |\mathrm{Cov}(WX, YZ) - \mathrm{Cov}(WX', YZ)  | \le 16 \| W\|_4  \| Y\|_4 \| Z\|_4 \| X - X'\|_4.\]
Recall the notations in \eqref{eq: Delta_g}. Take $S \subseteq I \setminus \{e\}$ and $T  \setminus \{e'\}$ where  $I = V(G) \cup E(G)$.
It follows from \eqref{eq:Delta_e_bd}, \eqref{eq:local_L_p_bd}, and \eqref{eq: compare_discrete_derivatives} that 
\[ \begin{gathered}
    \Big | \mathrm{Cov}( \Delta_e \Lambda, \Delta_e \Lambda^S, \Delta_{e'} \Lambda, \Delta_{e'} \Lambda^T  ) -  
 \mathrm{Cov}( \Delta_e \Lambda_{[e, R]} \Delta_e \Lambda^S_{[e, R]}, \Delta_{e'} \Lambda_{[e', R]} \Delta_{e'} \Lambda^T_{[e', R]}  ) \Big| \\ 
\le C'\|w_e\|_4^4 \Psi_G(R)^3 e^{-(4c_*/5)R}\\
\le C \|w_e\|_4^4 e^{-(c_*/5)R},
\end{gathered}
\]
where in the last inequality we used the assumption $G \in \mathsf{Growth}(A, c_*/5)$.

It follows from \eqref{eq:local_L_p_bd} and the fact that $\Delta_e \Lambda_{[e, R]} \Delta_e \Lambda^S_{[e, R]}$ only depends on the edge and vertex weights of the subgraph $\bB^e_{2R}$,
\[ |\mathrm{Cov}( \Delta_e \Lambda_{[e, R]} \Delta_e \Lambda^S_{[e, R]}, \Delta_{e'} \Lambda_{[e', R]} \Delta_{e'} \Lambda^T_{[e', R]}  )| \ \   \begin{cases}
     = 0 & \text{ if } \mathrm{dist}(e, e') \ge 4R+ 2, \\
     \le C \| w_e\|_4^4 \Psi_G(R)^4 & \text{ otherwise}.
\end{cases} 
\]
Therefore,
\begin{align*}
\sum_{e, e' \in E(G)} c(e, e') &\le C'  \| w_e\|_4^4 \Big ( |E(G)|^2  e^{-(c_*/5)R} + |E(G)| \Psi_G(4R+2) \Psi_G(R)^4 \Big) \\
&\le C \Big ( |E(G)|^2  e^{-(c_*/5) R} + |E(G)| \Psi_G(R)^8  \Big),
\end{align*}  
where we absorbed the moments of the weights into the multiplicative constant and used the fact  $\Psi_G(4R+2) \le C'' \Psi_G(R)^4 $ which is a consequence of the sub-multiplicativity of volume growth and bounded degree assumption.
It is not hard to convince ourselves that the same bound also holds for $\sum_{x, x' \in V(G)} c(x, x')$ and $\sum_{e \in E(G), x \in V(G)} c(e, x)$. Adding them, we arrive at the bound
\begin{equation} \label{eq:c_ij_bd}
    \sum_{i, j \in I} c(i, j) \le C  |E(G)|^2 e^{-cR} + C |E(G)| \Psi_G(R)^8.
\end{equation}
Now  \eqref{eq:Delta_e_bd} implies that
\[ \sum_{ e \in E(G)} \E | \Delta_e \Lambda |^3 \le C' \| w_e\|_3 ^3 |E(G)|
\le C  |E(G)|.  \]
Again the same bound holds for $\sum_{ x \in V(G)}\E |  \Delta_x \Lambda|^3.$ We combine them to  obtain
\begin{equation} \label{eq:Delta_i_bd}
\sum_{ i \in I} \E |  \Delta_i \Lambda |^3 \le C  |E(G)|.  
\end{equation}
The proof of Theorem~\ref{thm: no_of_dimers} is now complete after plugging in the estimates \eqref{eq:c_ij_bd} and \eqref{eq:Delta_i_bd} in Theorem~\ref{thm: no_of_dimers} along with variance lower bound $\Var(\Lambda) \geq c_0 |E(G)|$.

\subsection{Proof of Proposition~\ref{prop: var_lower_bound_dimers}}
In this section, we prove Proposition~\ref{prop: var_lower_bound_dimers}.
\begin{proof}
 For any smooth function $\varphi$ of $N$ i.i.d.\ standard Gaussian random variables, it is well known that, e.g., see \cite[page 59]{chatterjee2014superconcentration},
\[ \Var(\varphi) = \sum_{k=1}^N \frac{1}{k!}\sum_{1\le i_1, i_2, \ldots, i_k \le N}
(\E  [\partial_{x_1} \cdots \partial_{x_k} f])^2. \]
Considering only the terms with $k=1$  and applying Cauchy-Schwarz, we arrive at
the following lower bound on the variance of $\varphi$, 
 \begin{equation}\label{eq:lb_var_gaussian}
   \Var(\varphi) \ge N^{-1} 
\Big(\E  \Big[\sum_{i=1}^N \partial_{x_i}f \Big ]\Big)^2.
 \end{equation}
The above inequality also holds for any coordinate-wise differentiable function $\varphi$, and in particular,  for $\langle |\M_G| \rangle$ once we condition on the vertex weights $\nu = (\nu_x)_{x \in V(G)}$. Recall that the partial derivative of  $\langle |\M_G| \rangle$ with respect to the edge weight $w_e$ is given by 
\[ \partial_{w_e} \langle |\M_G| \rangle = \langle \1_{\{ e \in \M_G\} } |\M_G| \rangle - \langle 1_{\{ e \in \M_G\}} \rangle \langle |\M_G| \rangle. \]
Hence, from \eqref{eq:lb_var_gaussian} and Jensen's inequality, we obtain
\begin{align}
\Var (\langle |\M_G| \rangle) \ge \E_\nu \Var \big (\langle |\M_G| \rangle \mid  \nu \big) &\ge \E_\nu |E(G)|^{-1}  \big[\E \big( \langle |\M_G|^2 \rangle -  \langle |\M_G| \rangle^2 \mid \nu \big)\big]^2  \nonumber\\
&\ge |E(G)|^{-1} \big[\E \big( \langle |\M_G|^2 \rangle -  \langle |\M_G| \rangle^2 \big)\big]^2. \label{var_lb1}
\end{align}
Thus, we need to lower bound the annealed Gibbs variance of $|\M_G|$. We adapt arguments from \cite{kahn2000normal} where a linear lower bound on variance was provided for the size of a uniformly chosen (unweighted) matching of a bounded degree graph. 

Let $\gE$ and $\gVar$ denote the expectation and variance with respect to the Gibbs measure, which by gauge transformation (see, for instance, \cite[Lemma~1.11]{DeyKrishnan}), can be written as 
\[ \mu_G(M) \propto \exp \Big( \sum_{e \in M} \tilde w_e \Big), \ \ M \in \mathcal{M}_G, \]
where $\tilde w_e = w_e  - (\nu_x + \nu_y)$ for $e = (xy)$.
With this notation, the inequality \eqref{var_lb1} can be rewritten as
\begin{equation}\label{var_lb2}
   \Var (\langle |\M_G| \rangle)  \ge |E(G)|^{-1} \big[\E \gVar(|\M_G|)\big]^2.  
\end{equation}
We would like to lower bound $\E \gVar(|\M_G|)$. Note that a linear lower bound (in $|E(G)|$) is obtained in \cite[Corollary~5.10]{DeyKrishnan}, which relies on the Lee-Yang zeros and requires a strong assmuption on the weight distributions. We will obtain the same lower bound by using another approach, which is applicable to more general situations.

Let $E_K$ be the subset of edges of $G$ satisfying $\tilde w_e < K.$ Let $F \subseteq E_K$ be a matching (no two edges are incident on the same vertex). 

\noindent \textbf{Claim.} For any matching $F \subseteq E_K$,
\[ \gVar(|\M_G|) \ge (1+ e^K)^{-1}\gE|\M_G \cap F|.  \]
Let $\gH = \M_G \setminus F$ and let $F_\gH$ be
the set of edges of $F$ that do not meet any edges from $\gH$. 
Note that given $\gH$, by edge-disjointness, the conditional Gibbs distribution of $\M_G \cap F$ is the product Bernoulli measure on $F_\gH$ such that for each $e \in F_\gH$, the event $\{ e \in \M_G\}$  has Gibbs probability  $p_e:= e^{\tilde w_e}/ (1+ e^{\tilde w_e}).$ 
Then
\begin{equation}\label{Gibbs_var}
    \gVar(|\M_G|) \ge \gE \big[\gVar(|\M_G| \mid \gH)\big] = \gE \bigg[\sum_{ e \in F_\gH} p_e(1 - p_e)\bigg]. 
\end{equation}  
On the other hand, 
\begin{equation}\label{Gibbs_exp}
 \gE|\M_G \cap F| = \gE \big[\gE(|\M_G \cap F| \mid \gH)\big] = \gE \bigg[\sum_{ e \in F_\gH} p_e\bigg].  
 \end{equation}
Since for each $e \in E_K$, $e^{\tilde w_e} < e^K$, we have 
$1- p_e > (1+ e^K)^{-1}$. Therefore, the claim follows by comparing \eqref{Gibbs_var} and \eqref{Gibbs_exp}.

Since the graph $(V(G), E_K)$ has the maximum degree at most $D$, by Vizing's theorem, there exists a partition of the edge set $E_K$ into at most $D+1$ matchings $F_1, F_2, \ldots, F_{D+1}$. By Claim, 
\begin{align*}
 \gVar(|\M_G|) &\ge (1+ e^K)^{-1} \max_{ 1 \le i \le D +1} \gE |\M_G \cap F_i| \\
&\ge (1+ e^K)^{-1} \frac{1}{D +1} \sum_{i=1}^{D+1} \gE |\M_G \cap F_i| \\
&= (1+ e^K)^{-1} (D +1)^{-1} \gE |\M_G \cap E_K|.
\end{align*}
Therefore, by writing $c_1 = (1+ e^K)^{-1} (D +1)^{-1} $
\[ \E  \big[\gVar(|\M_G|)\big] \ge c_1 \sum_{ e \in E} \E [ \1_{\{ \tilde w_e < K \} } \gE  \1_{\{ e \in \M_G \} }   ]. \]
Let $Y_e$ be the union of the edge $e$ and all the edges of $G$ adjacent to $e$. Note that $|Y_e| \le 2 D +1$.
Suppose that $\tilde w_e  = \max_{ f \in Y_e} \tilde w_f \in [-K, K]$. Then there exists a positive constant $c_2>0$, depending only on $D$ and $K$, such that 
conditional on $\M \setminus Y_e$, 
\[ \gE [\1_{\{ e \in \M_G \} } \mid \M \setminus Y_e]  \ge c_2.\]
Finally, by choosing $K$ sufficiently large, depending on the vertex weight distribution, we have  
\[ \E  \big[\gVar(|\M_G|)\big] \ge c_1 c_2 \sum_{ e \in E} \prob ( \tilde w_e =  \max_{ f \in Y_e} \tilde w_f \in [-K, K]) \ge c |E(G)|, \] 
for some positive constant $c$, depending only on $D$ and $K$. By \eqref{var_lb2}, the proof of the lemma is now complete.
\end{proof}

\section{Open questions}\label{sec:open_prob}
We conclude the paper by listing a few open problems. 

\begin{itemize} 
\item Is it possible to remove the subexponential growth assumption on the underlying graphs and replace the finite $(2+\eps)$-moment assumption with just the finite second-moment assumption in Theorem~\ref{thm: CLT_two_moments}? 

\item It is interesting to see if the bounded degree assumption on the underlying 
graphs can be relaxed to include random graph models that are ``stochastically bounded degree", for example, Erd\"os-R\'enyi graphs $G(N, c/N)$ for constant $c>0$.

\item Does the CLT for the averaged number of dimers stated in Corollary~\ref{cor:clt_dimer} hold for graphs with exponential volume growth?  Also, can the variance lower bound of Proposition~\ref{prop: var_lower_bound_dimers} be extended for the non-Gaussian edge weights? 

\item Consider the disordered monomer-dimer model with  $\beta = \infty$, the so-called ``zero-temperature'' case. The free energy then reduces to the weight of the maximum weight matching on $G$, which is unique almost surely if the weights are assumed to be continuous. It is natural to ask whether this model exhibits a decay of correlation on any bounded degree graphs, say, on the finite boxes in $\Z^d$? The case when $G$ is a sparse Erd\H{o}s-R\'enyi graph or a random regular graph was studied in Gamarnik et.\ al.\ \cite{Gamarnik}, and a correlation decay was shown for exponentially distributed edge weights. Relying on this correlation decay result, Cao \cite{cao2021central} was able to prove the central limit theorem for the weight of the maximum weight matching in sparse Erd\H{o}s-R\'enyi graph with exponential edge weights. Beyond that,  any correlation decay result is unavailable for bounded degree graphs including non-regular trees, and the central limit theorem for the weight of maximum matching remains an open problem. 

\item Another related interesting model to look at would be the disordered pure dimer model on a finite box on $\mathbb{Z}^d$ with an even side-length. The configuration space here is the set of all perfect matchings (no monomers).  Does this model have  the correlation decay  at any temperature and at any dimension? What about  the central limit theorem for the free energy?  It might be worthwhile to mention that for $d=2$ or more generally for finite planar graphs, the weighted dimer model is known to be exactly solvable due to Kasteleyn \cite{kasteleyn1961statistics, kasteleyn1967graph}  and independently Temperley and Fisher \cite{temperley1961dimer}. Indeed, its partition function has an explicit formula in terms of Pfaffian (or determinant in case of bipartite graphs) of the Kasteleyn matrix. The formulas for correlation functions of the edges are also available due to Kenyon \cite{kenyon1997local}. For $d=2$, the dimer model with random weights (under some ellipticity assumption on the weights)  falls under the general framework considered by Berestycki, Laslier, and Ray \cite{berestycki2020dimers}  (see also \cite{ray2021quantitative}), where they showed that the fluctuation of the height function of a random dimer configuration converges to a Gaussian free field for a wide class of planar graphs. Their work implies a correlation decay for the dimer model. 
\end{itemize} 

\section*{Acknowledgement} We thank Gourab Ray for bringing the references \cite{berestycki2020dimers, ray2021quantitative} to our attention. We thank the anonymous referee for valuable comments and suggestions.  

\bibliographystyle{plain}
\bibliography{pos_temp}

\end{document}